\documentclass[12pt,twoside]{article}
\usepackage{pdfsync}
\usepackage{mathrsfs,amssymb,amsfonts,amsthm,amsmath,amsbsy}
\usepackage[numbers]{natbib}
\setcitestyle{square,aysep={},yysep={;}}
\usepackage{dsfont,verbatim,fancyhdr}
\usepackage{graphicx,float} 
\usepackage[hang]{caption} 
\usepackage[usenames,dvipsnames]{xcolor}
\usepackage[bookmarks,bookmarksnumbered,colorlinks=true,pdfstartview=FitV, linkcolor=Red,citecolor=blue!90!black,urlcolor=blue!90!black]{hyperref}
\usepackage[normalem]{ulem}
\usepackage{autonum}

\newif\ifpageone

\allowdisplaybreaks

\bibpunct{\textcolor{blue!90!black}{(}}{\textcolor{blue!90!black}{)}}{\textcolor{blue!90!black}{;}}{a}{\textcolor{blue!90!black}{,}}{\textcolor{blue!90!black}{;}}

\newcommand{\bs}[1]{\boldsymbol{#1}}
\def \d {\mathrm{d}}

\def \A {\mathcal{A}}
\def \B {\mathcal{B}}

\def \L {\mathcal{L}}

\def \E {\mathbb{E}}

\def \F {\mathcal{F}}

\def \U {\mathcal{U}}
\def \S {\mathcal{S}}

\def \Y {\mathcal{Y}}
\def \R {\mathbb{R}}

\def\Z{\mathbb{Z}}

\def \aa {\bs\alpha}
\def \Ga {\mathbf{Ga}}
\def \oo {\mathbf{0}}
\def \ee {\mathbf{e}}
\def \mm {\mathbf{m}}
\def \NN {\mathbf{N}}

\def \MM {\mathbf{M}}
\def \nn {\mathbf{n}}
\def \hh {\mathbf{h}}
\def \kk {\mathbf{k}}

\def \ii {\mathbf{i}}
\def \jj {\mathbf{j}}

\def \ll {\mathbf{l}}

\def \xx {\mathbf{x}}
\def \zz {\mathbf{z}}
\def \XX {\mathbf{X}}
\def \ZZ {\mathbf{Z}}
\def \yy {\mathbf{y}}
\def \YY {\mathbf{Y}}
	
\def \iid {\overset{\text{iid}}{\sim}}	
\def \nminus {\stackrel{\leftarrow}{\textbf{n}}}
\def \nplus {\stackrel{\rightarrow}{\textbf{n}}}

\newtheorem{theorem}{Theorem}[section]
\newtheorem{proposition}[theorem]{Proposition} 
\newtheorem{corollary}[theorem]{Corollary}
\newtheorem{lemma}[theorem]{Lemma} 
 
\newtheorem{definition1}[theorem]{Definition} 
 
\newtheorem{remark1}[theorem]{Remark} \newenvironment{remark}{\begin{remark1}\rm}{\hfill$\square$\end{remark1}}

\long\def\symbolfootnote[#1]#2{\begingroup\def\thefootnote{\hspace*{-1mm}\fnsymbol{footnote}}\footnote[#1]{#2}\endgroup}

\title{\bf 
Smoothing distributions for conditional Fleming--Viot and Dawson--Watanabe diffusions
}

\author{\normalsize
\textsc{Filippo Ascolani}\\
\emph{\normalsize Department of Decision Sciences and BIDSA, Bocconi University}\\[2mm]
\textsc{\normalsize Antonio Lijoi}\\
\emph{\normalsize Department of Decision Sciences and BIDSA, Bocconi University}\\[2mm]
\textsc{\normalsize Matteo Ruggiero\footnote{matteo.ruggiero@unito.it}}\\
\emph{\normalsize University of Torino and Collegio Carlo Alberto}}

\date{\today}

\begin{document}
\maketitle
\thispagestyle{empty}



\begin{abstract}
We study the distribution of the unobserved states of two measure-valued diffusions of Fleming--Viot and Dawson--Watanabe type, conditional on observations from the underlying populations collected at past, present and future times. If seen as nonparametric hidden Markov models, this amounts to finding the smoothing distributions of these processes, which we show can be explicitly described in recursive form as finite mixtures of laws of Dirichlet and gamma random measures respectively. We characterize the time-dependent weights of these mixtures, accounting for potentially different time intervals between data collection times, and fully describe the implications of assuming a discrete or a nonatomic distribution for the underlying process that drives mutations. In particular, we show that with a nonatomic mutation offspring distribution, the inference automatically upweights mixture components that carry, as atoms, observed types shared at different collection times. The predictive distributions for further samples from the population conditional on the data are also identified and shown to be mixtures of generalized P\'olya urns, conditionally on a latent variable in the Dawson--Watanabe case. 
\end{abstract}

\section{Introduction and summary of results}


\subsection{Problem statement}

We consider the problem of evaluating the unknown state of a Fleming--Viot (FV) or a Dawson--Watanabe (DW) diffusion \citep{W68,D75,FV79}, two of the most popular measure-valued models for stochastic population dynamics, conditional on partial information on the underlying population collected at past, present and future times. See \cite{EK86,EK93,D10,F10} for reviews on these processes.
This task relates to the literature on hidden Markov models \citep{CMR05} 
 and consists in determining the so-called \emph{smoothing distributions} of the marginal states $X_{t_{i}}$ of an unobserved Markov process, often called the hidden or latent \emph{signal}, evaluated at time $t_{i}$ given samples $Y_{t_{0}},\ldots,Y_{t_{N}}$ from the observation model, which is parameterised by the signal state, collected before and after $t_{i}$, where $0=t_{0}<\cdots<t_{N}$ and $0<i<N$. These conditional distributions are typically used to improve previous estimates obtained at a certain time once additional observations become available at later times, often resulting in a smoother estimated trajectory for the unobserved signal, and they also constitute the starting point for performing Bayesian inference on the model parameters (see, e.g., \citealp{KKK20}).
When all data are collected at the same time $t_{i}$, the problem reduces to a simple Bayesian update of the law of $X_{t_{i}}$ given the observations, with the unconditional law of $X_{t_{i}}$ acting as the prior on the signal actual location. In the case of the FV and DW diffusions we consider here, taken at stationarity, this simplifies to updating the law of a Dirichlet process and a gamma random measure, results which date back to the seminal works of \cite{F73,A74,L82,LW89}.
When only past and present data are considered, the smoothing  problem reduces to a \emph{filtering} problem, which evaluates the law of $X_{t_{i}}$ given $Y_{t_{0}},\ldots,Y_{t_{i}}$.

Closed form expressions for filtering and smoothing distributions can typically only be obtained for a handful of simple models even when the signal is finite-dimensional, e.g., by assuming linear Gaussian systems for both signal and observations or by taking a finite state space for the signal. In these cases the associated computations have complexity that grows linearly with the number of observations. Other modelling formulations, instead, typically require some form of approximation (see, e.g., \cite{DM98,Detal11}). Intermediate results can sometimes be obtained and entail computations whose complexity grows polynomially with the number of observations, yielding so-called \emph{computable filters} \cite{CG06}.
Along these lines, \cite{CG09} obtained a computable filter  for a two-type Wright--Fisher (WF) diffusion and \cite{PR14} for a $K$-type version, while \cite{PRS16} extended the latter to measure-valued signals driven by a FV or a DW diffusion, whereby the associated hidden Markov model can be regarded as nonparametric. Recently, \cite{KKK20} solved the smoothing problem for a $K$-type WF diffusion.

In this paper we characterize the laws of the marginal states of FV and DW diffusions, given samples from the respective underlying populations collected before and after the state temporal index, thus solving the smoothing problem for nonparametric hidden Markov models with these measure-valued unobserved signals. 
We show that these distributions can be written as finite mixtures of laws of Dirichlet and gamma random measures respectively, whose time-dependent mixture weights are fully described and can account for different time intervals between data collection times. As a byproduct of the above results, we describe the predictive distribution for further samples from the population given the entire dataset, which are shown to be mixtures of generalized P\'olya urns. Our results prove that computable smoothing and conditional sampling from the population are feasible with signals given by these FV and DW processes, bringing forward these models as possible canonical choices in a nonparametric framework for hidden Markov models.

To introduce our setting in more detail, let $\Y$ be a Polish space and $\alpha$ a finite and non-null Borel measure on $\Y$. The class of FV processes we consider here is the so-called infinitely-many-neutral alleles model \citep{EK86} with parent-independent mutation, whose transition function can be written \citep{EG93,Wea07}
\begin{equation}\label{FV transition}
P^{\text{FV}}_t(x,\d x') = \sum_{m = 0}^{\infty} d_m(t) \int_{\Y^m}\Pi_{\alpha + \sum_{i = 1}^m\delta_{y_i}}(\d x')x^{m}(\d y_1,\ldots,\d y_m).
\end{equation}
Here $x^{m}$ denotes an $m$-fold product measure, $d_m(t)$ is the probability that a $\mathbb{Z}_{+}$-valued death process starting at infinity with infinitesimal rates $n(\theta+n-1)/2$ from $n$ to $n-1$ is in state $m$ at time $t$ (see \citealp{G80,T84}), and 
 $\Pi_{\alpha}$ is the law of a Dirichlet random measure \citep{F73}, i.e., such that the projection over any Borel partition $A_{1},\ldots,A_{K}$ of $\Y$ has the Dirichlet distribution with parameters $\alpha(A_{1}),\ldots,\alpha(A_{K})$. Alternatively,
 \begin{equation}\label{DP law}
\Pi_{\alpha}(\cdot)=\mathbb{P}\bigg\{\sum_{i\ge1}p_{i}\delta_{V_{i}}\in \cdot\bigg\},\quad \quad 
(p_{1},p_{2},\ldots)\sim \textsc{PD}(\theta), \quad \quad 
V_{i}\iid P_{0},
\end{equation} 
where $\textsc{PD}(\theta)$ denotes a Poisson--Dirichlet distribution \citep{K75}, {with parameter $\theta:=\alpha(\Y)$, 
	and $P_{0}:=\alpha/\theta$}. Note that $\Pi_{\alpha}$ is the reversible measure of \eqref{FV transition}, while $P_{0}$, often called the mutant \emph{offspring distribution}, is the stationary distribution of the $\Y$-valued autonomous Markov process that drives mutations in the population (cf.~\citealp{EK93}). In this paper we will consider both cases of $P_{0}$ being discrete and nonatomic, the former reducing the above formulation to the so-called \emph{unlabeled} model \citep{EK81}, and we will show the different statistical implications that this choice determines in terms of smoothing.

Next we will consider the class of DW processes given by the branching measure-valued diffusions studied in \cite{EG93b}, which take values in the space of finite positive measures and has transition function
\begin{equation}\label{DW transition}
P^{\text{DW}}_t(z,\d z') = \sum_{m = 0}^{\infty} d_m^{|z|, \beta}(t) \int_{\Y^m}\Gamma^{\beta+S_t}_{\alpha + \sum_{i = 1}^m\delta_{y_i}}(\d z')(z/|z|)^m(\d y_1, \dots, \d y_m),
\end{equation}
where $|z|:=z(\Y)$ denotes, here and throughout, the total mass of $z$, 
\[
d_m^{|z|, \beta}(t) = \mbox{Po}(m\,|\, S_t|z|), \quad S_t = \beta/(e^{\beta t/2}-1),
\]
with $\text{Po}(\cdot|\lambda)$ a Poisson pmf with mean $\lambda$, and $\Gamma^{\beta}_{\alpha}$ is the law of a gamma random measure with shape $\alpha$ and rate $\beta>0$, i.e., a positive atomic measure $z$ such that for any finite collection of pairwise disjoint Borel subsets $A_{1},\ldots,A_K$ of $\Y$, the random variables $z(A_1),\ldots,z(A_K)$ are independent and $z(A_i)\sim \text{Gamma}(\alpha(A_i),\beta)$. 

FV processes have well-known countable representations in terms of limits of particle systems when the population size goes to infinity \citep{EK93,DK96}, and it is therefore a well established interpretation that they describe the evolving frequencies of alleles observed at a single locus in an ideally infinite population. On the other hand, DW processes arise as scaling limits of branching particle systems \citep{D93} and have a close relationship with FV processes, as, roughly speaking, the latter can be constructed as a normalised, unit-mass version of the former \citep{EM91, P91}.
It is therefore natural to see these models, and the associated sampling from the underlying particle representations, as dynamic extensions of classical statistical settings that are widely used in Bayesian nonparametric statistics. More precisely, denoting by $Y_{t}^{i}\in \Y$ the $i$-th observation collected at time $t$, for the first model we let 
\begin{equation}\label{FV HMM}
Y_{t}^i  |  X_{t} = x \overset{\text{iid}}{\sim} x, \quad X\sim \text{FV}_{\alpha}.
\end{equation}
Here $X\sim \text{FV}_{\alpha}$ denotes that $X$ is a FV process with transition \eqref{FV transition}, and given $X_{t}=x$, the observations are conditionally \emph{iid} random samples from the current underlying population summarised by the atomic probability measure $x$ (recall that \eqref{DP law} is supported by discrete probability measures). For the second model, we let
\begin{equation}\label{DW HMM}
\YY^{i}_{t}  |  Z_{t} = z \overset{\text{iid}}{\sim} \text{PP}(z), \quad Z \sim \text{DW}_{\alpha,\beta}.
\end{equation}
Here $Z \sim \text{DW}_{\alpha,\beta}$ denotes that $Z$ is a DW process with transition \eqref{DW transition}, and given $Z_{t}=z$, the $i$-th draw consists in a vector of observations from a Poisson point process with intensity $z$. That is, a number $m \sim \text{Po}(|z|)$ of random samples are collected from the normalised measure $\bar z:=z/|z|$, i.e.
\begin{equation}\label{PP1}
\begin{aligned}
&Y_{t,j}^i  |  Z = z, m \overset{\text{iid}}{\sim} \bar z, \quad j = 1, \dots, m, \quad \quad m  |  Z = z \sim \text{Po}(|z|).
\end{aligned}
\end{equation}
Both settings can be seen as nonparametric hidden Markov models, since the signals are measure-valued and the observations are conditionally independent given the signal state. 
In this framework, we are interested in determining the distributions of $X_{t}$ and $Z_{t}$ given observations from the respective models collected at times before and after $t$.


\subsection{Summary of results}

We informally summarize our main results on the smoothing distributions in the following theorem, and then briefly discuss its interpretation and implications.  We will denote throughout by $\L(X_{t})$ and $\L(Z_{t})$ the laws of $X_{t}$ and $Z_{t}$, and assume observations are taken at discrete times $0=t_{0}<t_{1}<\cdots<t_{N}=T$ where equidistance between collection times is not assumed. We also adopt here the compact notation $\YY_{t_{0:N}}$ for the entire dataset, postponing for the later sections the specification of its precise structure. 

\begin{theorem}\label{summary theorem}
Assume model \eqref{FV HMM} with $X$ being a FV diffusion with transition function \eqref{FV transition}, and that the data $\YY_{t_{0:N}}$ feature $K$  distinct values  $y_{1},\ldots,y_{K}$. Then for every $0\le i\le N$ there exists a set $\MM\subset \Z_{+}^{K}$ of finite cardinality and weights $\{w_{\mm},\mm\in \MM\}$ summing up to one such that 
\begin{equation}\label{summary theorem FV}
\L(X_{t_{i}}|\YY_{t_{0:N}})=\sum_{\mm\in \MM}w_{\mm}\Pi_{\alpha + \sum_{j = 1}^{K}m_{j}\delta_{y_{j}}}.
\end{equation} 
Similarly, assume model \eqref{DW HMM} with $Z$ being a DW diffusion with transition function \eqref{DW transition}, and that the data $\YY'_{t_{0:N}}$ feature $K'$ distinct values $y'_{1},\ldots,y'_{K'}$. Then for every $0\le i\le N$ there exists a set $\MM'\subset \Z_{+}^{K'}$ of finite cardinality, a constant $b\in \R_{+}$ and weights $\{w_{\mm}^{b},\mm\in \MM'\}$ summing up to one, such that 
\begin{equation}\label{summary theorem DW}
\L(Z_{t_{i}}| \YY'_{t_{0:N}})=\sum_{\mm\in \MM'}w^{b}_{\mm}\Gamma^{\beta+b}_{\alpha + \sum_{j = 1}^{K'}m_{j}\delta_{y'_{j}}}.
\end{equation} 
\end{theorem}

The previous results states that all smoothing distributions associated to conditional FV and DW diffusions, given samples from the respective past, present and future populations, are finite mixtures of laws of Dirichlet and gamma random measures respectively. 
The later sections will present all the details that characterize the elements displayed above, among which the precise form of the mixture weights, which have a recursive nature, and the fact that the mixture components carry only part of the information available from the dataset, through the atomic components 
$\sum_{j = 1}^{K}m_{j}\delta_{y_{j}}$ and $\sum_{j = 1}^{K'}m_{j}\delta_{y_{j}'}$ of the parameter measures.
An interesting implication of our results emerges when $\alpha$ is nonatomic, i.e., when the mutant offsprings in the population underlying the measure-valued diffusions are new types with probability one. In this case, the smoothing distributions become highly informative from an inferential perspective, and provide an automatic upweighting of the information which emerged at more than one collection time. Specifically, mixture components which carry types observed at multiple times (as atoms in the parameter measure) end up having weights of a greater order of magnitude with respect to the others. This provides valuable information on the population lineage structure and is consistent with the expected behaviour of a smoothing distribution, supposed to remove some of the noise incorporated in the previous estimates.

{Another  byproduct of our results is the availability of the conditional laws for further samples from the population at a certain time, e.g., in the context of \eqref{summary theorem FV}, the laws of $Y'_{t_{i}} |  \YY_{t_{0:N}}$, of $Y''_{t_{i}} |  \YY_{t_{0:N}},Y'_{t_{i}}$, and so on}. We identify these distributions, which extend results obtained in \cite{ALR20} and are shown for both models to take the form of mixtures of generalized P\'olya urn schemes with time-dependent weights, in the DW case conditional on a latent variable sampled from a mixture of negative binomial densities.

Note that a standard approach for obtaining the above smoothing distributions would in principle entail computing the marginal law of the signal at time $t_{n}$ via $\int \xi_{t_{n-1}}(\d x) P_{t_{n}-t_{n-1}}(x,\d x')$
(cf.~\eqref{forward operator} below)
where $P_{t}$ is the transition function of the signal and $\xi_{t_{n-1}}$ here stands for the law of the signal at time $t_{n-1}$ given some past data, the output of which must then be conditioned to future data points. These operations, carried out directly, are particularly troublesome for various reasons, the most prominent being the fact that the transition functions \eqref{FV transition} and \eqref{DW transition} are intractable (in fact $d_{m}(t)$ itself has an infinite series expansion; see \cite{JS17}). It is thus unfeasible to apply Bayes' Theorem directly for updating with new observations in this framework. Theorem \ref{summary theorem} instead shows that with the two models we consider, by carefully combining arguments that exploit the duality properties of the signals and the projective properties of the laws involved, the intractable measure given by conditioning the above integral to future observations can in fact be reduced to a finite mixture whose computation can therefore be carried out exactly.



\subsection{General notation and organisation of the paper}

We will take as sampling (or \emph{type}) space a complete and separable metric space $\Y$, so that the DW process takes values on the space of finite and non-null atomic measures on $\Y$ and the FV on its subspace of probability measures. We will denote the FV and DW states by $X_{t}$ and $Z_{t}$ respectively, while the corresponding bold capitals $\XX_{t}$ and $\ZZ_{t}$ will later denote the respective $K$-dimensional respective projections onto measurable partitions of $\Y$, at each fixed $t$. The main parameter for both \eqref{FV transition} and \eqref{DW transition} is a finite and non-null measure $\alpha$ on $\Y$, with total mass $\theta=\alpha(\Y)$ and normalized version $P_{0}=\alpha/\alpha(\Y)$. 
Integer vectors of multiplicities will be denoted by $\mm,\nn,\ldots$, with
$\mm = (m_1, \dots, m_K) \in \Z^K_+$ and $ |\mm| = \sum_{j = 1}^Km_j$, such that $\mm \leq \nn$ if and only if $m_j \leq n_j$ for every $j$. The symbol $\oo$ will denote the vector containing all zeros. These multiplicities will sometimes carry a subscript to indicate to what collection time they refer to, e.g., $\nn_{i}$ will be a shorthand notation for $(n_{t_{i},1},\ldots,n_{t_{i},K})$  referred to time $t_{i}$.
Similarly, observations from \eqref{FV HMM} collected at time $t_{i}$ will be denoted $\YY_{i}$ for $m$ samples $Y_{t_{i}}^{1},\ldots,Y_{t_{i}}^{m}$, with $m$ given. We will not index $\YY_{i}$ by the sample size $m$ as the information available will be encoded in the distinct types $y_{1},\ldots,y_{K}$ observed  and their respective multiplicities $\nn_{i}$. The notation for a collection of observations from model \eqref{DW HMM} is a bit more involved and will be described in detail in Section \ref{sec: DW prelim}.

The rest of the paper is organised as follows. In Section \ref{sec:operators} we define certain operators on measures which are useful for computing several quantities of interest for finite-dimensional signals. These will then be exploited in the subsequent sections for obtaining our results on the measure-valued signals. In Sections \ref{sec:FV} and \ref{sec:DW} we study the FV and DW cases respectively, where each section starts with preliminary results particularly focussed on the finite-dimensional model counterparts, to then derive the smoothing distributions in full detail and conclude by discussing the implications on conditional sampling of further observations from the population.


\section{Some operators on measures}\label{sec:operators}

To obtain the smoothing distributions for the two measure-valued signals discussed in the introduction, we are going to exploit the projective properties of the FV and the DW process. To this end, we need to set a few tools that ease notation and computation for the respective finite-dimensional counterparts. 

Consider a Markov process $\XX$ on $\R^{K}$, $K<\infty$, with transition function $P_{t}$ and initial distribution $\nu$. We assume $P_{t}$ is reversible with reversible measure $\nu_{0}$. In this section we regard $\XX$ as generic, but as anticipated above this setting will be used to model finite-dimensional projections of the measure-valued diffusions $X_{t}$ and $Z_{t}$. The dimension $K$ can therefore be thought of as representing the number of cells in which types in the population have been grouped or binned. Accordingly, \emph{iid} observations collected at time $t$ given $\XX_{t}=\xx_{t}$ generate multiplicities associated to the $K$ groups which can be encoded into a vector $\nn\in \Z_{+}^{K}$, whose associated density is $p(\nn|\xx_{t})$. We can then define the following operators acting on measures $\xi$:
\begin{itemize}
\item \emph{Update}: 
\begin{equation}\label{update operator}
\U_{\nn}(\xi)(\d \xx)
:=\frac{p(\nn|\xx)\xi(\d \xx)}{p_{\xi}(\nn)}
,\quad \quad 
p_{\xi}(\nn):=\int p(\nn|\xx)\xi(\d \xx).
\end{equation} 
This provides the conditional measure $\xi$ given observations with associated multiplicities $\nn$. It is analogous to Bayes' Theorem with $\xx$ acting as the random parameter and $\xi$ as its prior: hence, $\U_{\nn}(\xi)$ yields the posterior, whereas the denominator $p_{\xi}(\nn)$ is the marginal likelihood of $\nn$ obtained by integrating out the random parameter $\xx$. Here the multiplicities $\nn$ are observed at the same time $\xx$ refers to, as in \eqref{FV HMM} and \eqref{DW HMM}.

\item \emph{Forward propagation}: 
\begin{equation}\label{forward operator}
\F_{t}(\xi)(\d \xx'):=\xi P_{t}(\d \xx')=\int \xi(\d\xx)P_{t}(\xx, \d\xx').
\end{equation} 
This yields the unconditional measure of $X_{s+t}$ if $\xi$ is that of $X_{s}$, once the initial state is integrated out. 

\item \emph{Backward propagation}:
\begin{equation}\label{backward operator}
\B_{t}(\xi)(\d \xx'):=\xi Q_{t}(\d \xx')
=\int \xi(\d \xx)Q_{t}(\xx,\d \xx').
\end{equation} 
It is the {forward} propagation obtained by using the transition function of the time reversal of the signal, denoted here $Q_t$, and yields the unconditional measure of $X_{s-t}$ if $\xi$ is that of $X_{s}$, once $X_{s}$ is integrated out.


\end{itemize}

With a slight abuse of notation, when $\xi$ is dominated by a sigma-finite measure $\mu$ on $\R^{K}$, {we 
specialize} the previous operators as acting on densities, e.g., if $\xi(\d \xx)=f(\xx)\mu(\d \xx)$, then {$\U_{\nn}(f)(\xx)
	:=p(\nn|\xx)f(\xx)/p_{f}(\nn)$ and $p_{f}(\nn):=\int p(\nn|\xx)f(\xx)\mu(\d \xx)$,}
and similarly for \eqref{forward operator} and \eqref{backward operator}.

Note that the specific form of the backward transition is not necessary for our treatment, as we will leverage on Bayes' Theorem. See, e.g., Proposition \ref{prop: nonpar backward} below.

\begin{remark} 
Expanding on the above, and assuming all probability distributions of interest are dominated by $\mu$, one could define a \emph{smoothing operator} acting on two densities $f_{s},f_{u}$, indexed by $s<u$, by letting
\begin{equation}\label{smoothing_op}
\S_{s,t,u}^{\nn}(f_{s},f_{u}) (\xx) :=
C\ \F_{t-s}(f_s)(\xx) \B_{u-t}(f_{u})(\xx) \,\U_{\nn}(f_{0})(\xx)/f_{0}(\xx)^2, 
\end{equation}
for every $\xx$ such that $f_{0}(\xx)>0$, where $s<t<u$, $f_{0}$ is the density of $\nu_{0}$ with respect to $\mu$ and $C$ is a normalising constant that makes the left hand side a density. 
This yields the distribution of $\XX_{t}$, given observations at time $t$, if $X_{s}\sim f_{s}$ and $X_{u}\sim f_{u}$, obtained by jointly propagating $\XX_{s}$ forward of a $t-s$ interval, $\XX_{u}$ backward  of a $u-t$ interval, and then conditioning on observations collected at time $t$. 
The rationale of this operator can be outlined by considering that, for $t_{0}<t_{1}<t_{2}$, if $y_{t_{1}} |  \xx_{t_{1}}\sim p(y_{t_{1}} |  \xx_{t_{1}})$, we have
\[
\begin{aligned}
p(\xx_{t_{1}} | \xx_{t_{0}}, y_{t_{1}}, \xx_{t_{2}}) 
&\propto p(\xx_{t_{1}} | \xx_{t_{0}})p(y_{t_{1}} | \xx_{t_{1}}, \xx_{t_{0}})p(\xx_{t_{2}} | \xx_{t_{1}}, \xx_{t_{0}}, y_{t_{1}}) 
\\
&= p(\xx_{t_{1}} | \xx_{t_{0}})p(y_{t_{1}} | \xx_{t_{1}})p(\xx_{t_{2}} | \xx_{t_{1}})
\end{aligned}
\]
 where
 we have used the fact that conditionally on $\xx_{t_{1}}$, $y_{t_{1}}$ is independent on everything else, together with the Markov property.
By virtue of Bayes' Theorem, we now have {$p(\xx_{t_{2}} | \xx_{t_{1}})=p(\xx_{t_{2}})p(\xx_{t_{1}} | \xx_{t_{2}})/p(\xx_{t_{1}})$ and $p(\xx_{t_{1}} | y_{t_{1}})=p(\xx_{t_{1}})p(y_{t_{1}}| \xx_{t_{1}})/
p(y_{t_{1}})$ 
whereby the previous expression is proportional to $p(\xx_{t_{1}} | \xx_{t_{0}})p(\xx_{t_{1}} | \xx_{t_{2}})p(\xx_{t_{1}}| y_{t_{1}})/p(\xx_{t_{1}})^{2}.$}
This operator will not be essential for our calculations of the next sections, but it can provide a unified treatment of the previous operators applied at stationarity. In fact, we have the following special cases:
\begin{equation}\nonumber
\begin{aligned}
\S_{s,t,u}^{\oo}(f_{s},f_{0})=\F_{t-s}(f_s), \quad \quad 
\S_{s,t,u}^{\oo}(f_{0},f_{u})=\B_{u-t}(f_u),\quad \quad 
\S_{s,t,u}^{\nn}(f_{0},f_{0})=\U_{\nn}(f_{0}).
\end{aligned}
\end{equation} 
\end{remark}

Appropriate compositions of the above operators allow to represent all quantities of interest in this framework. For example {$p (\xx_{t_{i}}  |  \nn_{{i-1}}, \nn_{{i}}, \nn_{{i+1}} ) = \S^{\nn_{{i}}}_{t_{i-1},t_{i},t_{i+1}} (\U_{\nn_{i-1}}(f_{0}), \U_{\nn_{i+1}}(f_{0}) )$}
identifies the distribution of $\XX_{t_{i}}$ given observations at times $t_{i-1},t_{i},t_{i+1}$, obtained by first updating the stationary measure at times $t_{i-1},t_{i+1}$ given observations with multiplicities $\nn_{t-1},\nn_{t_{i+1}}$ respectively, then propagating both distributions to the intermediate time $t_{i}$, and finally updating the output of the last operation given the multiplicities observed at time $t_{i}$. This type of distributions will be instrumental for the proof of Theorem \ref{summary theorem}. 

\section{Fleming--Viot signals}\label{sec:FV}

\subsection{Preliminary results}\label{sec: FV preliminaries}

In this section we assume the signal is a FV process and the observation model \eqref{FV HMM} holds. 
A projection of the Dirichlet law in \eqref{DP law} onto a measurable partition $(A_1, \dots, A_K)$ of the sampling space $\Y$ yields a Dirichlet distribution with parameters $(\alpha(A_{1}),\ldots,\alpha(A_{K}))$, whose density with respect to the Lebesgue measure on the $(K-1)$-dimensional simplex, denoted $\pi_{\aa}(\xx)$, $\aa =  (\alpha(A_1), \dots , \alpha(A_K) )$, is proportional to $\xx^{\aa-\textbf{1}}:=x_1^{\alpha(A_1)-1} \dots \, x_K^{\alpha(A_K)-1}$. With a little abuse of notation, we will use the symbol $\pi_{\aa}$ to denote both the Dirichlet density and the corresponding measure. Similarly, a projection of a FV process with transition \eqref{FV transition} onto the same partition yields a Wright--Fisher diffusion, denoted $\text{WF}_{\aa}$, with transition function
\begin{equation}\label{WF transition}
P_t(\xx, \d \xx') = \sum_{m=0}^\infty d_m(t)\sum_{\mm \in \Z_+^K: \, |\mm| = m}\xx^\mm \binom{m}{\mm}\pi_{\aa+\mm}({\d \xx'}),
\end{equation}
and $d_{m}(t)$ as in \eqref{FV transition}, which has reversible distribution $\pi_{\aa}$. We will denote by $p_{t}(\cdot\,|\xx)$ the corresponding density function. 
In this scenario, \eqref{FV HMM} reduces to 
\begin{equation}\label{WF HMM}
Y_{t}^i|\XX_{t}=\xx\overset{\text{iid}}{\sim} \,\text{Categorical}(\xx), \quad \XX \sim \, \text{WF}_{\aa},
\end{equation} 
whereby for each $i$, $Y_{t}^{i}=j$ with probability $x_{j}$, for $j=1,\ldots,K$, and the update operator \eqref{update operator} yields the familiar Bayesian update for Dirichlet distributions $\U_{\nn}(\pi_{\aa})=\pi_{\aa+\nn}$.
It is useful to note for later reference that in \eqref{update operator} the marginal likelihood is
\begin{equation}\label{dirichlet categorical}
m(\nn) := p_{\pi_{\alpha}}(\nn)
= \frac{B(\aa+\nn)}{B(\aa)},\quad \quad 
B(\aa):=\frac{\prod_{j=1}^{K}\Gamma(\alpha_{j})}{\Gamma(|\aa|)}
\end{equation} 
often called Dirichlet-Categorical distribution. 
Define now 
\begin{equation}\label{h function}
h(\xx,\nn):=\frac{p(\nn|\xx)}{m(\nn)}
=\frac{B(\aa)}{B(\aa+\nn)}\xx^{\nn},\quad \quad 
\nn\in \Z_{+}^{K},
\end{equation} 
where $p(\nn|\xx)$ is the categorical likelihood in \eqref{WF HMM} expressed in terms of multiplicities of types. 
It will also be useful to note that for $\nn,\mm\in\Z_{+}^{K}$, we have
\begin{equation}\label{product_h}
h(\xx, \nn)h(\xx, \mm) = c(\nn,\mm)h(\xx, \nn+ \mm)
\end{equation} 
where
\begin{equation}\label{c(n,m)}
c(\nn,\mm)=\frac{m(\nn+\mm)}{m(\nn)m(\mm)}
=\frac{B(\aa)B(\aa+\mm+\nn)}{B(\aa+\nn)B(\aa+\mm)}.
\end{equation} 

Recall now that the WF diffusion is known to have moment-dual given by Kingman's \emph{typed} coalescent. More specifically, let $\MM_{t}$ be a death process on $\Z_{+}^{K}$ with rates {$\lambda_{\mm}=m_{j}(\theta+|\mm|-1)/2$} from $\mm$ to $\mm-\ee_{j}$, where $\ee_j$ is the canonical vector in direction $j$. Then the following duality identity
\begin{equation}\label{duality}
\E[h(\XX_{t},\mm)| \XX_{0}=\xx]
=
\E[h(\xx,\MM_{t})| \MM_{0}=\mm]
\end{equation} 
holds with $h$ as in \eqref{h function}. We will denote by 
$p_{\nn,\mm}(t)$ the transition probabilities of $\MM_t$ (cf.~\cite{PRS16}, Section 4.2). 
The above duality is used to prove the following result, needed later. 

\begin{lemma}
\label{lemma: predictive likelihood}
Let $\nn_{i+1}$ be the multiplicities observed at time $t_{i+1}$. Then
\begin{equation}\label{WF observation prediction}
p(\nn_{i+1}|\xx_{t_{i}})
=m(\nn_{i+1})\sum_{\oo\le\mm\le \nn_{i+1}}p_{\nn_{i+1},\mm}(t_{i+1}-t_{i})h(\xx_{t_{i}},\mm).
\end{equation} 
\end{lemma}
\begin{proof}
Using \eqref{h function}, we have
\begin{equation}\nonumber
\begin{aligned}
p(\nn_{i+1}|\xx_{t_{i}})
=&\,\int p(\nn_{i+1} |\xx_{t_{i+1}})p_{t_{i+1}-t_{i}}(\xx_{t_{i+1}}|\xx_{t_{i}})\d \xx_{t_{i+1}} \\
=&\,m(\nn_{i+1})\int h(\xx_{t_{i+1}},\nn_{i+1})p_{t_{i+1}-t_{i}}(\xx_{t_{i+1}}|\xx_{t_{i}})\d \xx_{t_{i+1}}\\
=&\,
m(\nn_{i+1})\E[h(\XX_{t_{i+1}},\nn_{i+1})| \XX_{t_{i}}=\xx_{t_{i}}],
\end{aligned}
\end{equation} 
where the integral is 
 over the $(K-1)$-dimensional simplex, from which \eqref{duality} leads to the result.
\end{proof}

The next Lemma formalizes the fact that a backward propagation, after a change of measure with respect to the stationary distribution, yields an analogous distributional result to a forward propagation, somehow carrying over the reversibility of FV processes to their conditional versions. In the following statement, the forward and backward operators $\F_{t},\B_{t}$ are applied to laws of FV states by extension of \eqref{forward operator}-\eqref{backward operator}, with $P_{t}$ as in \eqref{FV transition}.

\begin{proposition}
\label{prop: nonpar backward}
Assume \eqref{FV HMM}, let $y_{1},\ldots,y_{K}$ be distinct values and 
$\nn\in\Z_{+}^{K}$. Then\\ {$\F_{t}(\Pi_{\alpha+\sum_{j = 1}^{K}n_j\delta_{y_j}}) = 
	\B_{t}(\Pi_{\alpha+\sum_{j = 1}^{K}n_j\delta_{y_j}}),$}
with $\F_{t},\B_{t}$ as in \eqref{forward operator}-\eqref{backward operator}, 
and in particular
\begin{equation}\label{backward propagation FV}
\B_{t} \Big(\Pi_{\alpha+\sum_{j = 1}^{K}n_j\delta_{y_j}} \Big) 
=  \sum_{\oo \leq \kk \leq \nn} p_{\nn, \kk}(t)\Pi_{\alpha+\sum_{j = 1}^{K}k_j\delta_{y_j}}.
\end{equation} 
\end{proposition}
\begin{proof}
Denote by $\nn_{i-1}$ and $\nn_{i+1}$ the multiplicities of types for observations sampled at times $t_{i-1}$ and $t_{i+1}$ respectively in the setting of \eqref{WF HMM}. Then \cite{PR14} showed that
\begin{equation}\label{forward prediction}
\F_{t_{i}-t_{i-1}}(\pi_{\aa+\nn_{i-1}})(\xx_{t_{i}})
=
\sum_{\oo\le\kk\le \nn_{i-1}}p_{\nn_{i-1},\kk}(t_{i+1}-t_{i})\pi_{\aa+\kk}(\xx_{t_{i}}).
\end{equation} 
Furthermore, using \eqref{update operator} first and then Lemma \ref{lemma: predictive likelihood} we have
\begin{equation}\nonumber
\begin{aligned}
\B_{t_{i+1}-t_{i}}&\,(\pi_{\aa+\nn_{i+1}})(\xx_{t_{i}})
=p(\xx_{t_{i}}|\nn_{i+1})
= \frac{p(\xx_{t_{i}})p(\nn_{i+1}| \xx_{t_{i}})}{m(\nn_{i+1})}\\
=&\,p(\xx_{t_{i}})\sum_{\oo\le\kk \le \nn_{i+1}}p_{\nn_{i+1},\kk}(t_{i+1}-t_{i})h(\xx_{t_{i}},\kk)\\
=&\,\sum_{\oo\le\kk\le \nn_{i+1}}p_{\nn_{i+1},\kk}(t_{i+1}-t_{i})h(\xx_{t_{i}},\kk)p(\xx_{t_{i}})\\
=&\,
\sum_{\oo\le\kk\le \nn_{i+1}}p_{\nn_{i+1},\kk}(t_{i+1}-t_{i})\pi_{\aa+\kk}(\xx_{t_{i}})
\end{aligned}
\end{equation} 
where the last identity follows from from \eqref{h function} and \eqref{update operator}. By equating $\nn_{i-1}$ with $\nn_{i+1}$ and $t_{i}-t_{i-1}$ with $t_{i+1}-t_{i}$, one can now see that $\B_{t}(\pi_{\aa+\nn})=\F_{t}(\pi_{\aa+\nn})$, 
with $\B_{t}$ as in \eqref{backward operator}. The fact that $\F_{t}(\Pi_{\alpha+\sum_{j = 1}^{K}n_{j}\delta_{y_j}})$ equals the right hand side of \eqref{backward propagation FV} now follows from Theorem 3.1 in \cite{PRS16}, and the same proof can be used to show \eqref{backward propagation FV}, by seeing $\B_{t}(\pi_{\aa+\nn})$ as the projection of $\B_{t}(\Pi_{\alpha+\sum_{j = 1}^{K}n_j\delta_{y_j}})$ onto an arbitrary partition, from which the first statement also follows.
\end{proof}

Thanks to this equivalence between backward and forward propagation, we have that the same result of Lemma \ref{lemma: predictive likelihood} holds with $\nn_{i+1}$ replaced by $\nn_{i-1}$, i.e., referred to time $t_{i-1}$, leading to the expression obtained by replacing $t_{i+1}-t_{i}$ with $t_{i}-t_{i-1}$ in the right hand side of \eqref{WF observation prediction}.

To conclude the section, we can interpret \eqref{backward propagation FV} as providing a sort of \emph{distributional bridge} that may be of independent interest. The heart of the intuition can be easily appreciated by considering a single observation $y\in \Y$ collected at time $s$, whereby conditioning \eqref{DP law} on $y$ adds a point mass at $y$ to the parameter measure $\alpha$ and 
{$\B_{t} (\Pi_{\alpha+\delta_{y}}) 
= p_{1,0}(t)\Pi_{\alpha}+p_{1,1}(t)\Pi_{\alpha+\delta_{y}}.$}
Hence, an interval $t$ prior to observing $y$, the distribution of $X_{s-t}$ has law as in \eqref{DP law} which carries a point mass at $y$ with probability $p_{1,1}(t)$. This is the probability that a single lineage in the underlying coalescent is not lost by mutation in an interval $t$. Furthermore, since $p_{1,1}(t)=e^{-\theta t/2}\rightarrow 1$ as $t\rightarrow 0$, we have $\B_{t} (\Pi_{\alpha+\delta_{y_{0}}})\rightarrow \Pi_{\alpha+\delta_{y_{0}}}$ in total variation, i.e., the probability mass progressively concentrates on the component which carries the observed atom. More generally, we can state the following implication of Proposition \ref{prop: nonpar backward}, whose proof is omitted being a simple extension of the previous argument.

\begin{corollary}
In the context of Proposition \ref{prop: nonpar backward}, we have {$\B_{t} (\Pi_{\alpha+\sum_{i = 1}^{K}n_i\delta_{y_i}}) 
\rightarrow 
	\Pi_{\alpha+\sum_{i = 1}^{K}n_i\delta_{y_i}}$ in total variation as $t\to 0$}.
\end{corollary}

Hence, conditioning on the knowledge of future observations from the underlying population yields a law for the current FV state that assigns gradually increasing probability to mixture components that carry the point masses associated to the observed values.


\subsection{Smoothing distributions}\label{subsec: FV smoothing}

Using the results of the previous section, the characterization of the smoothing distribution for conditional FV processes will be provided in three steps. First, in Theorem \ref{FV one-step-smoothing}, we show that conditioning on observations collected at adjacent times yields a finite mixture of laws of Dirichlet random measures. Then, in Proposition \ref{FV smoothing weights}, we give a full description of the mixture weights for different choices of the offspring distribution $P_{0}$ (cf.~\eqref{DP law}). Finally, in Proposition \ref{thm: general smoothing FV}, we show how the general expression of Theorem \ref{summary theorem} can be obtained by recursive computation based on the previous results.

We denote by $\YY_{i-1},\YY_i,\YY_{i+1}$ vectors of observations collected as in \eqref{FV HMM} at times $t_{i-1}$, $t_i$, $t_{i+1}$ respectively, with  associated multiplicities $\nn_{i-1},\nn_{i},\nn_{i+1}$ for the distinct values $(y_1, \dots , y_K)$ observed overall.

\begin{theorem}\label{FV one-step-smoothing}
Under model \eqref{FV HMM}, let $\YY_{i-1},\YY_i,\YY_{i+1}$ be as above. Then there exist weights summing up to one, denoted $w_{\kk_{i-1},\nn_{i}, \kk_{i+1}}(\Delta_i, \Delta_{i+1})$, for $\kk_{i-1}\le\nn_{i-1},\kk_{i+1}\le \nn_{i+1}$, such that
\begin{equation}\label{one step FV}
\begin{aligned}
\L(X_{t_{i}}&\, |  \YY_{i-1}, \YY_i, \YY_{i+1})=\\
=&\,\sum_{\oo \leq \kk_{i-1} \leq \nn_{i-1}}\sum_{\oo \leq \kk_{i+1} \leq \nn_{i+1}}
w^{\nn_{i-1},\nn_{i+1}}_{\kk_{i-1},\nn_{i}, \kk_{i+1}}(\Delta_{i}, \Delta_{i+1})
\Pi_{\alpha_{\kk_{i-1}+\nn_{i}+\kk_{i+1}}},
\end{aligned}
\end{equation} 
where $\Delta_{i} = t_{i}-t_{i-1}$, $\Delta_{i+1} = t_{i+1}-t_{i}$, and 
\begin{equation}\label{alpha conditional}
\alpha_{\kk_{i-1}+\nn_{i}+\kk_{i+1}}
=\alpha+\sum_{j = 1}^{K}(k_{i-1,j}+n_{i,j}+k_{i+1,j})\delta_{y_{j}}.
\end{equation} 
\end{theorem}
\begin{remark}
In the previous result, we have used notation $\kk_{i-1}$ and $\kk_{i+1}$ for the integrating variables, whose indices should help the intuition by indicating the time point they refer to. Note however that in principle these quantities are determined at time $t_i$, being the number of lineages in a time-reversed genealogy. Instead of using generic integrating variables $\ii,\jj$, we choose to adopt this notational convention here and later for the sake of readability.
\end{remark}
\begin{proof}
Without loss of generality, let $i = 1$ and denote $X_{1}=X_{t_{1}}$. Given a measurable partition $\A = (A_1, \dots , A_m)$ of  $\mathcal{Y}$, let $X_1(\A) := (X_1(A_1), \dots , X_1(A_m))$ and denote by $\YY(\A)$ the list of labels derived from binning $\YY$ into $\A$, i.e., whose $i$-th element is $j$ if $Y_{i}\in A_{j}$. 
Further, let $\{\B_n,n\ge1\} = \{(B^n_1, \dots, B_n^n),n\ge1\}$ be a sequence of increasingly finer partitions of $\Y$ such that $\B_n$ is finer than $\A$, for every $n$, and such that $\max_{j} \text{diam}(B^n_j) \to 0$ as $n$ diverges. Since $\B_n$ is increasingly finer, we have that $\left(\mathbb{E}\left[ f(X_1(\mathcal{A})) | \YY_0(\mathcal{B}_n),\YY_1(\mathcal{B}_n), \YY_2(\mathcal{B}_n) \right] \right)_n$ is a martingale for every bounded and continuous function $f$ (see Proposition $V.2.7$ in \cite{Cinlar}). Thus, by the martingale convergence theorem we have that $X_1(\mathcal{A}) | \YY_0(\mathcal{B}_n),\YY_1(\mathcal{B}_n), \YY_2(\mathcal{B}_n)$ converges weakly to $X_1(\mathcal{A}) | \YY_0, \YY_1, \YY_2$ as $n \to \infty$. The left hand side of the previous expression can be characterized, {by virtue of de Finetti's Theorem, 
in terms of the 
predictive distributions} of $Y_1^{1:k}(\A) | \YY_0(\B_n),\YY_1(\B_n), \YY_2(\B_n)$ for arbitrary $k$, where $Y_1^{1:k}(\A)$ denotes $k$ samples from $X_1(\A)$. Without loss of generality, let now $n$ be large enough that different observations lie in different sets of $\B_n$, and write, for brevity, $X_{1,n} := X_1(\B_n)$ and $\YY_{i,n} := \YY_i(\B_n)$. Let also $p(y^{1:k} | \YY_{0,n},\YY_{1,n}, \YY_{2,n}  )$ be the density of the vector $Y^{1:k}_1(\A)$ evaluated at $y^{1:k}$, conditional on the binned observation $\YY_{0,n},\YY_{1,n}, \YY_{2,n}$. 
Then we have
\begin{equation}\label{start}
\begin{aligned}
p  ( y^{1:k} &| \YY_{0,n},\YY_{1,n}, \YY_{2,n}  ) \propto p  ( y^{1:k}, \YY_{0,n},\YY_{1,n}, \YY_{2,n}  )= \E \left[ p  ( y^{1:k} ,\YY_{0,n},\YY_{1,n}, \YY_{2,n} | X_{1,n}  ) \right]\\
& = \E \left[ p  ( y^{1:k} | X_{1,n}  ) p  ( \YY_{0,n} | X_{1,n}  ) p  ( \YY_{1,n} | X_{1,n}  ) p  (\YY_{2,n} | X_{1,n}  ) \right] 
\end{aligned}
\end{equation}
where in the last identity we have used the conditional independence of the observations given the signal state (cf.~\eqref{FV HMM}). By Lemma \ref{lemma: predictive likelihood} and the subsequent comment we get
\[
\begin{aligned}
&p  ( \YY_{0,n} | X_{1,n}  ) \propto \sum_{\oo\le\kk_{i-1}\le \nn_{i-1}}p_{\nn_{i-1},\kk_{i-1}}h (X_{1,n},\kk_{i-1} ) \\
& p  ( \YY_{2,n} | X_{1,n}  ) \propto \sum_{\oo \le \kk_{i+1} \le \nn_{i+1}}p'_{\nn_{i+1},\kk_{i+1}}h (X_{1,n},\kk_{i+1} )\\
& p  ( \YY_{1,n} | X_{1,n}  ) \propto h ( X_{1,n}, \nn_{i}  )
\end{aligned}
\]
where
$p_{\nn_{i-1},\kk_{i-1}} := p_{\nn_{i-1},\kk_{i-1}}(\Delta_{i})$ and $p'_{\nn_{i+1},\kk_{i+1}} := p_{\nn_{i+1},\kk_{i+1}}(\Delta_{i+1})$. By linearity and by \eqref{product_h}, \eqref{start} is thus proportional to
\begin{equation}\nonumber
\begin{aligned}
\sum_{\oo\le\kk_{i-1}\le \nn_{i-1}}
&\,\sum_{\oo \le \kk_{i+1} \le \nn_{i+1}}p_{\nn,\kk_{i-1}}p'_{\mm,\kk_{i+1}}
\frac{m^{(n)} ( \kk_{i-1} + \nn_{i}+ \kk_{i+1}  )}{m^{(n)} (\kk_{i-1}  )m^{(n)} (\nn_{i} )m^{(n)} (\kk_{i+1} )}\\
&\,\times \E \left[ p  ( y^{1:k} | X_{1,n}  ) h  ( X_{1,n}, \kk_{i-1}+\nn_{i}+\kk_{i+1}  ) \right],
\end{aligned}
\end{equation} 
where $m^{(n)}$ denotes the marginal distribution in \eqref{dirichlet categorical} relative to the model induced by the partition $\B_{n}$.
Moreover, using \eqref{h function} and \eqref{update operator} it can be seen that 
\begin{equation}\nonumber
\E\Big[p(\nn'| \xx)h(\xx,\nn)\Big]
=\int p(\nn'| \xx)\frac{p(\nn|\xx)}{m(\nn)}p(\xx)\d \xx
=\frac{m(\nn,\nn')}{m(\nn)}=m_{\nn}(\nn'),
\end{equation} 
with $m_{\nn}(\nn'):=p(\nn'|\nn)$,
hence we can write
\[
\E \left[ p  ( y^{1:k} | X_{1,n}  ) h  ( X_{1,n}, \kk_{i-1}+\nn_{i}+\kk_{i+1}  ) \right] = m_{\kk_{i-1}+\nn_{i}+\kk_{i+1}} ( y^{1:k}  ).
\]
Note that the above identity holds since $\B_n$ is finer than $\A$. Hence the left hand side of \eqref{start} equals
\begin{equation}\nonumber
\begin{aligned}
\sum_{\oo\le\kk_{i-1}\le \nn_{i-1}}
&\,\sum_{\oo \le \kk_{i+1} \le \nn_{i+1}}p_{\nn_{i-1},\kk_{i-1}}p'_{\nn_{i+1},\kk_{i+1}}\\
&\,\times\frac{m^{(n)} ( \kk_{i-1} + \nn_{i}+ \kk_{i+1}  )}{C_{n}m^{(n)} (\kk_{i-1}  )m^{(n)} (\nn_{i} )m^{(n)} (\kk_{i+1} )}
m_{\kk_{i-1}+\nn_{i}+\kk_{i+1}} ( \nn' )
\end{aligned}
\end{equation} 
where $C_n$ is a normalizing constant and $\nn'$ is the vector of multiplicities associated to $y^{1:k}$. Since $m_{\kk_{i-1}+\nn_{i}+\kk_{i+1}} ( \nn' )$ is the distribution induced by the P\'olya Urn scheme of the Dirichlet--multinomial model, it follows that the law of $Y^{1:k}_1(\A) | \YY_{0,n},\YY_{1,n}, \YY_{2,n}$ is exchangeable. Note in particular that this marginal distribution does not depend on the partition $\B_{n}$, since $\A$ is a coarser partition.  Given the arbitrariness of $k$, we can appeal to de Finetti's Theorem to conclude that the law of $X_1(\A) | \YY_{0,n}, \YY_{1,n}, \YY_{2,n}$ is given by
\begin{equation}\label{WF smoothing}
\begin{aligned}
\sum_{\oo\le\kk_{i-1}\le \nn_{i-1}}
&\,\sum_{\oo \le \kk_{i+1} \le \nn_{i+1}}
p_{\nn_{i-1},\kk_{i-1}}p'_{\nn_{i+1},\kk_{i+1}}\\
&\,\times\frac{m^{(n)} ( \kk_{i-1} + \nn_{i}+ \kk_{i+1}  )}{C_{n}m^{(n)} (\kk_{i-1}  )m^{(n)} (\nn_{i} )m^{(n)} (\kk_{i+1} )}
 \pi_{\alpha(\A)+ \kk_{i-1}(\A)+\nn_{i}(\A)+ \kk_{i+1}(\A)}
\end{aligned}
\end{equation} 
where $\alpha(\A)=(\alpha(A_{1}),\ldots,\alpha(A_{m}))$ and $\kk_{i-1}(\A),\nn_{i}(\A),\kk_{i+1}(\A)$ denote the multiplicities projected onto $\A$. The limit as $n\rightarrow \infty$ can now be computed by virtue of the martingale convergence theorem. The proof is completed by observing that the limiting weights do not depend on the partition $\B_{n}$ and the previous display coincides with the projection onto the partition $\A$ of a finite mixture of laws of Dirichlet processes. 
\end{proof}

The previous result provides an explicit representation of the conditional law of a FV state given observations at adjacent times, but does not investigate in full detail the mixture weights, denoted generically in the Theorem statement, which we do next. To pursue this task, we need to compute
\begin{equation}\label{weights_FV}
\lim_{n \to \infty}
\frac{m^{(n)} ( \kk_{i-1} + \nn_{i}+ \kk_{i+1}  )}
{C_{n}m^{(n)} (\kk_{i-1}  )m^{(n)} (\nn_{i} )m^{(n)} (\kk_{i+1} )}.
\end{equation}
since the term $p_{\nn_{i-1},\kk_{i-1}}p'_{\nn_{i+1},\kk_{i+1}}$ in \eqref{WF smoothing} does not depend on $n$. 
In the setting of Theorem \ref{FV one-step-smoothing}, denote now by
\begin{equation}\nonumber
\begin{aligned}
D_{i-1} =&\, \left\{j \in \{1, \dots , K\}\, : \, n_{i-1,j} > 0 \text{ and either }  n_{i,j} > 0 \text{ or } n_{i+1,j} > 0\right\},\\
D_{i+1} =&\, \left\{j \in \{1, \dots , K\}\, : \, n_{i+1,j} > 0 \text{ and either }  n_{i,j} > 0 \text{ or } n_{i-1,j} >0 \right\},
\end{aligned}
\end{equation} 
the set of distinct values in $\YY_{i-1}$ shared with $\YY_{i}$ or $\YY_{i+1}$,  and those in $\YY_{i+1}$ shared with $\YY_{i-1}$ or $\YY_{i}$, respectively. Then
\[
\mathcal{D} = \left\{(\kk, \kk') \leq (\nn_{i-1},\nn_{i+1}):\, k_{j} > 0, \, k_{j'}' > 0,\forall j \in D_{i-1} \text{ and } j' \in D_{i+1}\right\}
\]
is the set of multiplicities $(\kk, \kk')$ not greater than $(\nn_{i-1},\nn_{i+1})$ such that the frequency of distinct values shared between different collection times is strictly positive. 
For example, if $t_{i}$ is the current time index, suppose we have {$\nn_{i-1} = (1,3,0)$, $\nn_i = (0,0,1)$ and $\nn_{i+1} = (0,2,1)$},
whereby of the three types observed overall, at time $t_{i-1}$ we observed multiplicities 1 and 3 for the first two, at time $t_{i}$ an instance of a third type, and so on. Then
\begin{equation}\nonumber
\begin{aligned}
\mathcal{D} 
=&\, \left\{(\ii, \jj):\, \ii\le (1,3,0), \, \jj\le(0,2,1), \, i_2 > 0, j_2 > 0, j_3 > 0 \right\}
\end{aligned}
\end{equation} 
is given by vectors of multiplicities not greater than $(\nn_{i-1},\nn_{i+1})$, with positive entries for type two, which is shared by times $t_{i-1},t_{i+1}$, and for type three, limited to the second coordinate, since it is shared by time $t_{i+1}$ and the current time. In other words, multiplicities not greater than those observed, with positive entries for types: (i) observed at both times different from the current, or (ii) observed at the current time and at least another time. Notice that $\mathcal{D} = \emptyset$ corresponds to the case in which no values are shared between the three collection times $t_{i-1}$, $t_i$ and $t_{i+1}$, which holds, for example, when all the observations are distinct.

Before stating the result, note that when $P_{0}$ is supported by a countably infinite set, $\mm(\nn)$ can be defined by extension of \eqref{dirichlet categorical}, where all but a finite number of terms simplify in the ratio. 
Let also $a^{(b)} = a(a+1)\dots(a+b-1)$ denote the Pochhammer symbol.

\begin{proposition}\label{FV smoothing weights}
In the setting of Theorem \ref{FV one-step-smoothing}, let $\tilde p=p_{\nn_{i-1},\kk_{i-1}}(\Delta_{i})p_{\nn_{i+1},\kk_{i+1}}(\Delta_{i+1})$. Then
\begin{itemize}
\item[A.] if $P_0$ is discrete,
\begin{equation}\label{FV weight P_0 discrete}
w^{\nn_{i-1},\nn_{i+1}}_{\kk_{i-1}, \nn_{i}, \kk_{i+1}}(\Delta_{i}, \Delta_{i+1}) \propto 
\tilde p\,
\frac{m(\kk_{i-1} +\nn_{i}+ \kk_{i+1})}{m( \kk_{i-1} )m( \nn_{i} )m(\kk_{i+1})};
\end{equation} 
\item[B.] if $P_0$ is nonatomic and $\mathcal{D} = \emptyset$, 
\[
w^{\nn_{i-1},\nn_{i+1}}_{\kk_{i-1}, \nn_{i}, \kk_{i+1}}(\Delta_{i}, \Delta_{i+1}) \propto 
\tilde p\,
\frac{\theta^{(|\kk_{i-1}|)}\theta^{(|\kk_{i+1}|)}}{(\theta+|\nn_{i}|)^{(|\kk_{i-1}|+|\kk_{i+1}|)}};
\]
\item[C.] if $P_0$ is nonatomic and $\mathcal{D} \neq \emptyset$, 
\begin{equation}\nonumber
\begin{aligned}
w^{\nn_{i-1},\nn_{i+1}}_{\kk_{i-1}, \nn_{i}, \kk_{i+1}}(\Delta_{i}, \Delta_{i+1}) \propto&\,
\tilde p\,
\frac{\theta^{(|\kk_{i-1}|)}\theta^{(|\kk_{i+1}|)}}{(\theta+|\nn_{i}|)^{(|\kk_{i-1}|+|\kk_{i+1}|)}}\prod_{j =1}^K{\frac{(k_{i-1,j}+n_{i,j}+k_{i+1,j}-1)!}{(k_{i-1,j}-1)!\,(n_{i,j}-1)!\,(k_{i+1,j}-1)!} }
\end{aligned}
\end{equation} 
if $(\kk_{i-1}, \kk_{i+1}) \in \mathcal{D}$, and zero otherwise.
\end{itemize}
\end{proposition}
\begin{proof}
Statement $A$ follows from the fact that ultimately the limit partition sets with positive multiplicities will be those coinciding with the support points of $P_{0}$. 

Assume now, without loss of generality, that the partition $\B_n$ is such that the first observation lies in $B_1^n$, the second in $B_2^n$ and so on. The density of a vector $\kk$ of multiplicities is in this case determined by the Blackwell--MacQueen P\'olya urn scheme \citep{BM73} to be
\[
m^{(n)}(\kk) = \frac{\prod_{j = 1}^K\prod_{h = 0}^{k_j-1}\left(\theta P_0(B_j^n)+h\right)}{\theta ^{(|\kk|)}},
\]
with the convention that $\prod_{h = 0}^{-1} = 1$. Denoting $m_{\nn}(\nn'):=p(\nn'|\nn)$, as in the proof of Theorem \ref{FV one-step-smoothing}, it follows that
\[
\begin{aligned}
&\,\frac{m^{(n)}(\kk_{i-1} +\nn_{i}+ \kk_{i+1})}{m^{(n)}( \kk_{i-1} )m^{(n)}( \nn_{i} )m^{(n)}(\kk_{i+1})} 
= \frac{m^{(n)}_{\nn_i}(\kk_{i-1} + \kk_{i+1})}{m^{(n)}( \kk_{i-1} )m^{(n)}(\kk_{i+1})} \\
&\,= \frac{\theta^{(|\kk_{i-1}|)}\theta^{(|\kk_{i+1}|)}}{(\theta+|\nn_{i}|)^{(|\kk_{i-1}|+|\kk_{i+1}|)}}
\prod_{j = 1}^K\frac{\prod_{h = 0}^{k_{i-1,j}+k_{i+1,j}-1}\left(\theta P_0(B_j^n)+n_{i,j}+h\right)}{\prod_{h = 0}^{k_{i-1,j}-1}\left(\theta P_0(B_j^n)+h\right)\prod_{h = 0}^{k_{i+1,j}-1}\left(\theta P_0(B_j^n)+h\right)}.
\end{aligned}
\]
If $\mathcal{D} = \emptyset$, since no values are shared across times, we have that, for every $j$, at most one between $k_{i-1,j}$ and $k_{i+1,j}$ is non zero, and in such case we have $n_{i,j} = 0$.  Then, if $k_{i-1, j} > 0$ we have $k_{i+1, j} = 0$ and $n_{i,j} = 0$, so
\[
\begin{aligned}
&\frac{\prod_{h = 0}^{k_{i-1,j}+k_{i+1,j}-1}\left(\theta P_0(B_j^n)+n_{i,j}+h\right)}{\prod_{h = 0}^{k_{i-1,j}-1}\left(\theta P_0(B_j^n)+h\right)\prod_{h = 0}^{k_{i+1,j}-1}\left(\theta P_0(B_j^n)+h\right)} 
 = 1,
\end{aligned}
\]
and the same happens when $k_{i+1,j}>0$,  $k_{i-1, j} = 0$ and $n_{i,j} = 0$. This leads to statement $B$. 

If $\mathcal{D} \ne \emptyset$, since some values are shared across times,  there exists a $j$ such that one of the following is true: (i) $n_{i,j} > 0$ and $k_{i-1,j} > 0$; (ii)  $n_{i,j} > 0$ and $k_{i+1,j} > 0$; (iii) $k_{i-1,j} > 0$ and $k_{i+1,j} > 0$. In case (i)
\[
\frac{\prod_{h = 0}^{k_{i-1,j}+k_{i+1,j}-1}\left(\theta P_0(B_j^n)+n_{i,j}+h\right)}{\prod_{h = 0}^{k_{i-1,j}-1}\left(\theta P_0(B_j^n)+h\right)\prod_{h = 0}^{k_{i+1,j}-1}\left(\theta P_0(B_j^n)+h\right)} \to \infty
\]
as $n\rightarrow \infty$, since $P_0(B_j^n) \to 0$ and the denominator vanishes. Case (ii) is obtained similarly. In case (iii), rewrite the weights as
\begin{equation}\label{rewrite_weights}
\begin{aligned}
&\frac{\prod_{h = 0}^{k_{i-1,j}+k_{i+1,j}-1}\left(\theta P_0(B_j^n)+n_{i,j}+h\right)}{\prod_{h = 0}^{k_{i-1,j}-1}\left(\theta P_0(B_j^n)+h\right)\prod_{h = 0}^{k_{i+1,j}-1}\left(\theta P_0(B_j^n)+h\right)} \\
&= \frac{\prod_{h = 0}^{k_{i-1,j}-1}\left(\theta P_0(B_j^n)+n_{i,j}+h\right)}{\prod_{h = 0}^{k_{i-1,j}-1}\left(\theta P_0(B_j^n)+h\right)}\frac{\prod_{h = k_{i-1,j}}^{k_{i-1,j}+k_{i+1,j}-1}\left(\theta P_0(B_j^n)+n_{i,j}+h\right)}{\prod_{h = 0}^{k_{i+1,j}-1}\left(\theta P_0(B_j^n)+h\right)}.
\end{aligned}
\end{equation}
Here the left factor is such that
\[
\frac{\prod_{h = 0}^{k_{i-1,j}-1}\left(\theta P_0(B_j^n)+n_{i,j}+h\right)}{\prod_{h = 0}^{k_{i-1,j}-1}\left(\theta P_0(B_j^n)+h\right)} \geq 1
\]
and the right factor can be written
\[
\frac{\prod_{h = k_{i-1,j}}^{k_{i-1,j}+k_{i+1,j}-1}\left(\theta P_0(B_j^n)+n_{i,j}+h\right)}{\prod_{h = 0}^{k_{i+1,j}-1}\left(\theta P_0(B_j^n)+h\right)} = \frac{\prod_{h =0}^{k_{i+1,j}-1}\left(\theta P_0(B_j^n)+k_{i-1,j}+n_{i,j}+h\right)}{\prod_{h = 0}^{k_{i+1,j}-1}\left(\theta P_0(B_j^n)+h\right)}.
\]
Therefore, the left hand side of  \eqref{rewrite_weights} is {greater than or equal to}
\[
\frac{\prod_{h = 0}^{k_{i+1,j}-1}\left(\theta P_0(B_j^n)+k_{i-1,j}+n_{i,j}+h\right)}{\prod_{h = 0}^{k_{i+1,j}-1}\left(\theta P_0(B_j^n)+h\right)}
\] 
which diverges to infinity as $n\rightarrow \infty$ as well. 
Thus, nodes with shared observations have divergent unnormalized weights. Let $\mathcal{S} = D_{i-1} \cup D_{i+1}$ be the set of shared values and let $(\kk_{i-1}, \kk_{i+1}) \in \mathcal{D}$. Then, we can write the associated weight as
\[
\begin{aligned}
& \frac{\theta^{(|\kk_{i-1}|)}\theta^{(|\kk_{i+1}|)}}{(\theta+|\nn_{i}|)^{(|\kk_{i-1}|+|\kk_{i+1}|)}}\\
&\quad \times
\prod_{j = 1}^K
\frac{\prod_{h = 0}^{k_{i-1,j}+n_{i,j}+k_{i+1,j}-1}\left(\theta P_0(B_j^n)+h\right)}{\prod_{h = 0}^{k_{i-1,j}-1}\left(\theta P_0(B_j^n)+h\right)\prod_{h = 0}^{n_{i,j}-1}\left(\theta P_0(B_j^n)+h\right)\prod_{h = 0}^{k_{i+1,j}-1}\left(\theta P_0(B_j^n)+h\right)}\\
&
= \frac{\theta^{(|\kk_{i-1}|)}\theta^{(|\kk_{i+1}|)}}{(\theta+|\nn_{i}|)^{(|\kk_{i-1}|+|\kk_{i+1}|)}}\\
&\quad \times
\prod_{j \in \mathcal{S}}
\frac{\prod_{h = 1}^{k_{i-1,j}+n_{i,j}+k_{i+1,j}-1}\left(\theta P_0(B_j^n)+h\right)}{\prod_{h = 1}^{k_{i-1,j}-1}\left(\theta P_0(B_j^n)+h\right)\prod_{h = 1}^{n_{i,j}-1}\left(\theta P_0(B_j^n)+h\right)\prod_{h = 1}^{k_{i+1,j}-1}\left(\theta P_0(B_j^n)+h\right)}\\
&\quad \times \frac{1}{\prod_{j \in \mathcal{S}}P_0(B_j^n)\prod_{h = 0}^{n_{i,j}-1}P_0(B_j^n)}.
\end{aligned}
\]
Here the third factor on the right hand side is common to each node and is cancelled upon normalizing, while the second factor converges to the product in statement $C$, proving the result.	
\end{proof}

We now have a full description of \eqref{one step FV}, which is a finite mixture of laws of Dirichlet random measures whose parameter measure $\alpha + \sum_{j = 1}^{K}(k_{i-1,j}+n_{i,j}+k_{i+1,j})\delta_{y_{j}}$ contains, besides the unnormalised offspring measure $\alpha$, the current observations $\YY_i$ and a subset of the observations $(\YY_{i-1}, \YY_{i+1})$ collected at adjacent times. The mixture weights are in turn determined by the following two elements. The first is given by the transition probabilities of the death process associated to the typed coalescent, which determines the probability that past and future data are atoms in the respective random measures as a function of the distance between time $t_{i}$ and the adjacent times. As the lags $\Delta_{i}$ and $\Delta_{i+1}$ grow, the number of survived lineages is lower with higher probability and the random measures in the mixture will carry, on average, less information in terms of types observed at different times. 
The second element is the joint marginal likelihood of past, present and future data. For instance, when the offspring distribution is discrete, the ratio in \eqref{FV weight P_0 discrete} is higher when $m(\kk_{i-1} +\nn_{i}+ \kk_{i+1}) > m(\kk_{i-1}) m(\nn_{i})m(\kk_{i+1})$, i.e. when sampling jointly $\kk_{i-1}$, $\nn_{i}$ and $\kk_{i+1}$ has high probability relative to sampling them separately: this provides a smoothing effect by favouring multiplicities $\left( \kk_{i-1}, \nn_{i},\kk_{i+1}\right)$ with the same types collected at different times. Such mechanism comes to an extreme when the offspring distribution is nonatomic. In this case, the weights of the mixture components that do not carry atoms observed at multiple times, in the sense of the set $\mathcal{D}$, vanish in the limit. 

The above results do not include the case of $\alpha$ having both a continuous and discrete component. The same tools used for proving Proposition \ref{FV smoothing weights} can in principle be used to deal with this case as well, where we expect parts A, B and C of the statement to hold for the respective parts of the parameters measure. In essence, values drawn by the discrete part of $\alpha$ are subjected to the probability that lineages survive as controlled by the term $\tilde p$ (hence ultimately by the death process), whereas values drawn from the continuous part of $\alpha$ are also, in addition, subjected to whether they are shared across collection times. A full description of such results would require a cumbersome notation and would not add further valuable insight, hence, we will not pursue this task here. 

We can now revisit the first statement of Theorem \ref{summary theorem} in light of the above findings. 
Let $\YY_{t_{0:N}}$ be the entire dataset sampled in model \eqref{FV HMM}, and let $K$ be the number of distinct values in $\YY_{t_{0:N}}$. Denoting by $\nminus_{i-1}$ the total multiplicities of the vector $\YY_{t_{0:i-1}}$, we know there exist weights $\{v_{1,\kk_{i-1}} \}$ such that
\begin{equation}\label{xi note}
\begin{aligned}
\L  ( X_{t_{i}}  |  \YY_{t_{0:i-1}}) = \sum_{\kk_{i-1} \leq \nminus_{i-1}}v_{1,\kk_{i-1}}\Pi_{\alpha+\sum_{j = 1}^K k_{i-1,j}\delta_{y_j}}.
\end{aligned}
\end{equation} 
This can be obtained recursively starting from $X_{t_{0}}\sim \Pi_{\alpha}$, where the reversible measure $\Pi_{\alpha}$ acts as prior (or unconditional distribution) for the marginal state of the signal. Upon observing $\YY_{t_{0}}$, the update yields {$\L  ( X_{t_{0}}  |  \YY_{t_{0}}  ) = \Pi_{\alpha+\sum_{j = 1}^{K_{0} }n_{0,j}\delta_{y_j}}$,}
where {$\nn_{0}$ 
	is the vector of multiplicities} of the $K_{0}$ distinct values in $\YY_{t_{0}}$. Propagating forward the previous through $\F_{t}$, one obtains
\begin{equation}\nonumber
\L  ( X_{t_{1}}  |  \YY_{t_{0}}  ) = \sum_{\kk_{0}\le \nn_{0}}p_{\nn_{0},\kk_{0}}(t_{1}-t_{0})\Pi_{\alpha+\sum_{j = 1}^{K_{0}}k_{0,j}\delta_{y_j}},
\end{equation} 
which can then be updated once data in $t_{1}$ become available by observing that \eqref{update operator} satisfies
\begin{equation}\nonumber
\U_{\nn}\bigg(\sum_{i=1}^{H}w_{i}\xi_{i}\bigg)
=\sum_{i=1}^{H}\frac{w_{i}p_{\xi_{i}}(\nn)}{\sum_{h=1}^{H}w_{h}\xi_{h}(\nn)}\U_{\nn}(\xi_{i}).
\end{equation} 
Proceeding in this way, alternating updates and forward propagations, leads to \eqref{xi note}; see \cite{PRS16}, Section 3.1, for further details.
Denoting now by $\nplus_{i+1}$ the total multiplicities of the vector $\YY_{t_{i+1:N}}$, by virtue of Proposition \ref{prop: nonpar backward} and of the linearity of \eqref{backward operator}, there exist weights $\{v_{2,\kk_{i+1}} \}$ such that
\begin{equation}\label{xi two}
\L  ( X_{t_{i}}  |  \YY_{t_{i+1:N}}  ) = \sum_{\kk_{i+1} \leq\nplus_{i+1}}v_{2,\kk_{i+1}}\Pi_{\alpha+\sum_{j = 1}^K k_{i+1,j}\delta_{y_j}}.
\end{equation} 
This can also be {obtained 
by working} backwards from $X_{t_{N}}\sim \Pi_{\alpha}$, then updating given the multiplicities $\nn_{N}$ of the $K_{N}$ distinct values in $\YY_{t_{N}}$, which yields {$\L  ( X_{t_{N}}  |  \YY_{t_{N}}  ) = \Pi_{\alpha+\sum_{j = 1}^{K_{N} }n_{N,j}\delta_{y_j}}$,}
then, using Proposition \ref{prop: nonpar backward}, propagating backwards to get
\begin{equation}\nonumber
\L  ( X_{t_{N-1}}  |  \YY_{t_{N}}  ) = \sum_{\kk_{N}\le \nn_{N}}p_{\nn_{N},\kk_{N}}(t_{N}-t_{N-1})\Pi_{\alpha+\sum_{j = 1}^{K_{N}}k_{N,j}\delta_{y_j}},
\end{equation} 
and so on. The following proposition connects the two above distributions to yield the general representation in Theorem \ref{summary theorem}.


\begin{proposition}\label{thm: general smoothing FV}
Assume \eqref{xi note} and \eqref{xi two} hold, and let $\nn_{i}$ be the vector of multiplicities of $\YY_{t_{i}}$ relative to the $K$ distinct values in the whole dataset $\YY_{t_{0:N}}$. Then 
\begin{equation}\label{FV general smoothing two sums}
\L( X_{t_{i}} | \YY_{t_{0:N}} ) = 
\sum_{\kk_{i-1}\le \nminus_{i-1}}\sum_{\kk_{i+1}\le \nplus_{i+1}}
p_{\kk_{i-1},\nn_{i},\kk_{i+1}}
\Pi_{\alpha_{\kk_{i-1}+\nn_{i}+\kk_{i+1}}},
\end{equation} 
where $\alpha_{\kk_{i-1}+\nn_{i}+\kk_{i+1}}$ is as in \eqref{alpha conditional} and the weights 
\begin{equation}\label{FV pesi general smoothing two sums}
p_{\kk_{i-1}, \nn_{i},\kk_{i+1}}= 
\sum_{\hh\le \nminus_{i-1}:\,\hh \ge \kk_{i-1}}
\sum_{\ll\le \nplus_{i+1}:\, \ll \geq \kk_{i+1}}
v_{1,\hh}v_{2,\ll}w^{\hh, \ll}_{\kk_{i-1}, \nn_i,\kk_{i+1}},
\quad 
\kk_{i-1}\le \nminus_{i-1}, \kk_{i+1}\le \nplus_{i+1},
\end{equation} 
with $w^{\hh, \ll}_{\kk_{i-1}, \nn_i,\kk_{i+1}}$ as in Proposition \ref{FV smoothing weights}, sum up to one.
\end{proposition}
\begin{proof}
The first statement follows as in the proof of Theorem \ref{FV one-step-smoothing} by noting that conditioning on $\YY_{i-1},\YY_{i+1}$ in \eqref{one step FV} is qualitatively analogous to conditioning to $\YY_{t_{0:i-1}},\YY_{t_{i+1,N}}$ in \eqref{xi note}-\eqref{xi two}, since the main argument \eqref{start} is based on the factorization of the likelihoods of the data collected prior, concurrently and after the signal state. The second statement follows by the linearity of the expected value in \eqref{start} and by readjusting the weights.
\end{proof}
The proof of Proposition \ref{thm: general smoothing FV} clarifies that the smoothing mixture is computed in two steps: first $\L( X_{t_{i-1}} | \YY_{t_{0:i-1}} )$ and $\L( X_{t_{i+1}} | \YY_{t_{i+1:N}} )$ are computed through backward and forward filtering respectively, then the smoothing operator is applied, as in Theorem \ref{FV one-step-smoothing}. The first step leads to two mixtures whose number of components is $\prod_{k = 1}^K(1+\sum_{j = 0}^{i-1}n_{t_j, k} )$ and $\prod_{k = 1}^K(1+\sum_{j = i+1}^{N}n_{t_j, k} )$
respectively, as shown in Section $4$ of \cite{KKK20}. Recall that here $K$ is the number of distinct values observed in the entire dataset, which is considered as given. 
Since each distinct element of the smoothing distribution is now given by a distinct choice of $\kk_{i-1}$ and $\kk_{i+1}$, the total number of components in the smoothing distribution is therefore 
\begin{equation}\label{number_of_components}
\prod_{k = 1}^K\bigg(1+\sum_{j = 0}^{i-1}n_{t_j, k} \bigg)\bigg(1+\sum_{j = i+1}^{N}n_{t_j, k} \bigg).
\end{equation}
As expected, smoothing comes at a greater nominal computational cost than filtering, since, roughly speaking, it combines information from both past and future. However, the actual cost of smoothing is expected to be much lower than the nominal, due to two factors. The first, specific to the current modelling assumptions, is that in the scenario of Statement C in Proposition 3.5, with a continuous baseline distribution, the number of components is automatically pruned by the smoothing operator, which discards values that are not shared across times. Hence \eqref{number_of_components} represents a crude upper bound. The second factor is that some mixture component weights are typically negligible. This aspect, which had already been noted in \cite{CG06} and was investigated in detail for Wright--Fisher and Cox--Ingersoll--Ross models in \cite{KKK20}, suggests various possible pruning strategies that allow to approximate the smoothing distribution, lowering the actual computational cost by some order of magnitudes while keeping a high precision in the approximation.


%


\subsection{Predictive distributions}\label{subsec: FV pred}

As a corollary to Proposition \ref{thm: general smoothing FV}, we can derive the predictive distribution of further samples collected at time $t_i$, given the original data set $\YY_{t_{0:N}}$. The following statement extends the results in \cite{ALR20}.

\begin{corollary}\label{cor:FV predictive}
In the setting of Proposition \ref{thm: general smoothing FV}, let \eqref{FV general smoothing two sums} be the conditional law of $X_{t_{i}}$ given $\YY_{t_{0:N}}$. Then the law of the $(k+1)$th further sample $Y^{k+1}$ from $X_{t_{i}}$ is
\begin{equation}\label{FV predictive}
\begin{aligned}
\mathbb{P}(&\,Y^{k+1} \in A  |  \YY_{t_{0:N}} , Y^{1:k})=\\
=&\,
\sum_{\kk_{i-1}\le \nminus_{i-1}}\sum_{\kk_{i+1}\le \nplus_{i+1}}
p_{\kk_{i-1},\nn_{i},\kk_{i+1}}
\frac{
\alpha_{\kk_{i-1}+\nn_{i}+\kk_{i+1}}(A)
+ \sum_{j=1}^{k}\delta_{Y^{j}}(A)}{\theta+ |\kk_{i-1}|+|\nn_{i}| + |\kk_{i+1}|+k}
\end{aligned}
\end{equation} 
for every Borel set $A$  of $\Y$, with $\alpha_{\kk_{i-1}+\nn_{i}+\kk_{i+1}}$ is as in \eqref{alpha conditional} and $p_{\kk_{i-1},\nn_{i},\kk_{i+1}}$ as in Proposition \ref{thm: general smoothing FV}.
\end{corollary}
\begin{proof}
The statement can be easily proved by noting that 
\[
\begin{aligned}
\mathbb{P}  ( Y_i \in A  |  \YY_{t_{0:N}}  ) =& \E \left[ \mathbb{P}  ( Y_i \in A  |  X_i, \YY_{t_{0:N}}  )  |  \YY_{t_{0:N}} \right] = \E \left[ X_i(A)  |  \YY_{t_{0:N}} \right]
\end{aligned}
\]
and using \eqref{FV general smoothing two sums}, after recalling that if $W\sim \Pi_{\alpha}$ with $\Pi_{\alpha}$ as in \eqref{DP law} then $\E(W)=\alpha/\alpha(\Y)$.
\end{proof}

Here for brevity we have used the notation $Y^{k+1}$ for the additional $(k+1)$-st sample instead of the correct notation $Y^{|\nn_{i}|+k+1}$, given the original dataset already contains $|\nn_{i}|$ observations sampled at $t_{i}$.
Recall that the predictive distribution for observations sampled from a Dirichlet random measure, a realization of \eqref{DP law}, is described by the Blackwell--MacQueen P\'olya urn scheme \citep{BM73}, whereby for $\alpha=\theta P_{0}$, 
\begin{equation}\nonumber
Y_{1}\sim P_{0}, \quad \quad 
Y_{k+1}|Y_{1},\ldots,Y_{k}\sim \frac{\alpha+\sum_{j=1}^{k}\delta_{Y_{j}}}{\theta+k}.
\end{equation} 
It is then clear that \eqref{FV predictive} is a finite mixture of generalized P\'olya urn schemes, whose sampling mechanism can be described as follows. For each $k\ge1$, given we already observed the further sample $Y^{1:k}$, 
\begin{list}{
$\bullet$
}{\itemsep=1mm\topsep=2mm\itemindent=-3mm\labelsep=2mm\labelwidth=0mm\leftmargin=9mm\listparindent=0mm\parsep=0mm\parskip=0mm\partopsep=0mm\rightmargin=0mm\usecounter{enumi}}
\setcounter{enumi}{0}
\item choose a pair $(\kk_{i-1},\kk_{i+1})$ with probability $p_{\kk_{i-1},\nn_{i},\kk_{i+1}}$
\item draw a categorical random variable $J\in\{1,2,3\}$ with probabilities proportional to $\theta, |\kk_{i-1}|+|\nn_{i}| + |\kk_{i+1}|$ and $ k$ respectively
\item draw
\begin{equation}\nonumber
Y^{k+1} \sim 
\left\{
\begin{array}{ll}
P_{0},   &  \quad \text{if }J=1, \\[2mm]
\sum_{j = 1}^{K}\frac{k_{i-1,j}+n_{i,j}+k_{i+1,j}}
{|\kk_{i-1}|+|\nn_{{i}}|+|\kk_{i+1}|}\delta_{y_{j}},   &  \quad \text{if }J=2, \\[2mm]
\frac{1}{k}\sum_{j=1}^{k}\delta_{Y^{j}},  &  \quad \text{if }J=3,
\end{array}
\right.
\end{equation} 
where $(y_{1},\ldots,y_{K})$ in the second expression are the distinct values in $\YY_{t_{0:N}}$.
\end{list}

We conclude the section with the observation that \eqref{FV general smoothing two sums} is the de Finetti measure of the sequence $\{Y^{|\nn_{i}|+k},k\ge1\}$, i.e.  {as $k\to\infty$}
\begin{equation}\nonumber
\mathbb{P}(Y^{|\nn_{i}|+k+1} \in \cdot  \,|  \YY_{t_{0:N}} , Y^{|\nn_{i}|+1:|\nn_{i}|+k})
\Rightarrow
P^{*}
\quad \quad \text{a.s.}
\end{equation} 
where $\Rightarrow$ denotes weak convergence and $P^{*}$ is the law in \eqref{FV general smoothing two sums}. This can be 
proved along the same lines of Proposition 5 in \cite{ALR20}, where we also refer the reader for an application of a special case of \eqref{FV predictive}.


\section{Dawson--Watanabe signals}\label{sec:DW}

\subsection{Preliminary results}\label{sec: DW prelim}

Assume now the signal is a DW process and the observation model \eqref{DW HMM} holds. A projection of a gamma random measure $\Gamma_{\alpha}^{\beta}$ onto a measurable partition $A_1, \dots , A_K$ of $\Y$ yields independent gamma random variables $\Gamma_{\alpha}^{\beta}(A_{j})\sim \text{Ga}(\alpha(A_{j}),\beta)$. Similarly, a projection of a DW process with transition \eqref{DW transition} onto the same partition yields independent subcritical continuous-state branching processes (CSBPs), denoted as $\text{CSBP}(\alpha(A_{j}),\beta)$, which can also be seen as Cox--Ingersoll--Ross processes and have transition function
\[
p_t(z_j, \d z_j') = \sum_{m = 0}^{\infty} \text{Po}  (m  |  z_j S_t  )\text{Ga}  ( \d z_j'  |  \alpha(A_j) + m, \beta + S_t  ), \quad\quad  
S_t = \frac{\beta}{e^{\beta t/2}-1};
\]
cf.~\cite{PRS16}.
Therefore, the projection of model \eqref{DW HMM} yields
\begin{equation}\label{K-CIR HHM}
\begin{aligned}
\YY_{t_n}^i  |  \ZZ_t = \zz \overset{\text{iid}}{\sim} \text{PP}(\zz), \quad \ZZ \sim \mathrm{CSBP}_K(\aa,\beta),
\end{aligned}
\end{equation}
where $\text{CSBP}_K(\aa,\beta)$ denotes a vector of $K$ independent CSBPs as above, $\aa =  (\alpha(A_1), \dots , \alpha(A_K) )$ and $\text{PP}(\zz)$ a Poisson point process with intensity $\zz=(z(A_{1}),\ldots,z(A_{K}))$. Cf.~\eqref{PP1}. In other words, conditionally on $\ZZ_t = \zz$, the number $m$ of observations is sampled from $\text{Po}(|\zz|)$ and a vector $\YY_{t}^{1}=(Y_{t,1}^{1}, \dots, Y_{t,m}^{1})$ is drawn independently with multinomial distribution with parameters $\zz/|\zz|$. In this setting we will refer to a single \emph{draw} as the collection of such $m$-sized vector $\YY_{t}^{1}$, with $m$ random. Consider now a number $c\in \mathbb{N}$ of such draws, later referred to as the \emph{cardinality} of the vector, denoted $\YY_{t}^{1:c} = (\YY_{t}^{1}, \dots , \YY_{t}^{c})$ with multiplicities $(\nn_{t,1}, \dots , \nn_{t,c})$ associated to the distinct values $(y_1, \dots , y_K)$ observed overall. We denote by $\NN_{t}^{1:c} = \sum_{j = 1}^c\nn_{t,j}$ the vector that collects the overall multiplicities observed for the distinct values. 
The following Lemma identifies the marginal distribution of the observations.

\begin{lemma}\label{lemma_marginalDW}
Consider a sample $\YY$ with multiplicity $\NN$ and cardinality $c$ from  \eqref{K-CIR HHM}. Then the marginal likelihood is
\begin{equation}\label{K-CIR marg lkh}
r_{c}(\YY) =  \bigg(\frac{\beta}{\beta+c} \bigg)^{\theta} \bigg(\frac{1}{\beta+c}  \bigg)^{|\NN|}\frac{\Gamma (\theta+ |\NN|  )}{\Gamma (\theta  )\prod_{j = 1}^c\prod_{k = 1}^K\Gamma (n_{j,k} + 1)}m ( \NN  )
\end{equation} 
with $m(\NN)$ as in \eqref{dirichlet categorical} and $\theta = \alpha(\Y)$.
\end{lemma}
\begin{proof}
Recall that $\zz^{\mm-\textbf{1}} = z_1^{m_1-1}\dots \, z_K^{m_K-1}$ for $\mm = (m_1, \dots, m_K)$.  Notice that
\[
p\left(\YY_t^j \mid \zz \right) = \text{Po}(|\nn_j| \, ; \, |\zz|)\text{Mult}(\nn_j \, ; \, |\nn_j|, \zz/|\zz|) = \frac{e^{-|\zz|}|\zz|^{|\nn_j|}}{|\nn_j|!} \, \binom{|\nn_j|}{\nn_j} \left(\frac{\zz}{|\zz|} \right)^{\nn_j} = \frac{e^{- |\zz|}}{\prod_{k = 1}^Kn_{j,k}!}\zz^{\nn_j}
\]
Then, writing $\NN=\NN_{t}^{1:c}$, we have
\[
\begin{aligned}
r_{c}(\YY)  
=&\, \int p(\YY | \zz)p(\zz)\d \zz 
= \int \frac{e^{-c|\zz|}}{\prod_{j = 1}^c\prod_{k = 1}^Kn_{j,k}!}\zz^{\NN}\frac{\beta^\theta}{\prod_{j = 1}^K\Gamma(\alpha_j)}\zz^{\aa-1}e^{-\beta |\zz|} \, \d\zz \\
=&\, \frac{\beta^\theta}{\prod_{j = 1}^c\prod_{k = 1}^Kn_{j,k}!\prod_{j = 1}^K\Gamma(\alpha_j)} \int z^{\aa+ \NN -1}e^{-(\beta+c)\zz} \, \d\zz \\
=&\,  \bigg(\frac{\beta}{\beta+c} \bigg)^{\theta} \bigg(\frac{1}{\beta+c}  \bigg)^{|\NN|}
\frac{1}{\prod_{j = 1}^c\prod_{k = 1}^Kn_{j,k}!}\prod_{j = 1}^K \frac{\Gamma(\alpha_j+ |\NN_j|)}{\Gamma(\alpha_j)} 
\end{aligned}
\]
whereby using now \eqref{dirichlet categorical} gives the result. 
\end{proof}
Note that the special case $r_{1}(\YY)$ is given by the product of the marginal distributions associated to Gamma--Poisson and Dirichlet--Multinomial models, the former being a negative binomial distribution with success probability $1/(\beta+1)$ and number of failures $\theta$, evaluated at $|\NN|$. 
Thus, denoting
\begin{equation}\label{gamma_function}
\gamma_{a}(n) =   \bigg(\frac{\beta}{\beta+a} \bigg)^{\theta} \bigg(\frac{1}{\beta+a}  \bigg)^{n}\frac{\Gamma (\theta+ n  )}{\Gamma (\theta  )}, \quad n \in \mathbb{N}, a > 0,
\end{equation}
{one can rewrite \eqref{K-CIR marg lkh} as} 
\begin{equation}\label{marginal_equality}
r_{c}(\YY) = \frac{\gamma_{c}  (| \NN|  )}{\prod_{j = 1}^c\prod_{k = 1}^K\Gamma (n_{j,k} + 1)}m ( \NN  )
\end{equation}
with $m ( \NN  )$ as in \eqref{dirichlet categorical}. In the following, with an abuse of notation, we will use $\NN$ interchangeably with $\YY$, that is $r_c(\NN)$ and $p(\NN \mid \zz)$ in place of $r_c(\YY)$ and $p(\YY \mid \zz)$: indeed, even if $\NN$ does not store the counts $|\nn_j|$ associated to each draw, the main quantities to be defined afterwards will depend on $\YY$ only through $\NN$ itself.

In this setting, the action of the update operator \eqref{update operator} is a straightforward extension of the update of a Gamma prior given Poisson observations, yielding here {$\U_{\NN} (\Ga(\aa, \beta)  ) = \Ga  (\aa+\NN, \beta+c  ).$}
Proceeding now similarly to Section \ref{sec: FV preliminaries}, define
\begin{equation}\label{h_functionCIR}
h(\zz, \NN, c) = \frac{p(\NN  |  \zz)}{r_{c}(\NN)} = e^{-c|\zz|} \bigg( \frac{\beta + c}{\beta}  \bigg)^{\theta}(\beta+c)^{|\NN|}\prod_{j = 1}^K \frac{\Gamma(\alpha_j)}{\Gamma(\alpha_j + \NN_j)}\zz^\NN,
\end{equation}
whereby we note that 
\begin{equation}\label{h p = post}
h(\zz, \NN, c)p(\zz)=p(\zz| \NN,c)
\end{equation} 
and that the product of functions in \eqref{h_functionCIR} satisfies the identity 
\begin{equation}\label{h_productCIR}
h(\zz, \NN, c)h(\zz, \MM, d) = \frac{r_{c+d}(\NN+ \MM)}{r_{c}(\NN)r_{d}(\MM)}h(\zz, \NN+\MM, c+d).
\end{equation}
Theorem 4.1 in \cite{PRS16} showed the following duality relation for $h$ functions as in \eqref{h_functionCIR},
\begin{equation}\label{K-CIR duality}
\E \left[ h(\ZZ_t, \mm, c)  |  \ZZ_0 = \zz \right] = \E \left[h(\zz, \MM_t, C_t)  |  \MM_0 = \mm, C_0 = c \right],
\end{equation}
where $\ZZ_{t}$ is as in \eqref{K-CIR HHM}, $\MM_t$ is a death process on $\mathbb{Z}_+^K$ that jumps from $\mm$ to $\mm-\ee_i$ at rate $2\mm_i(\beta+C_t)$, and $C_t$ is a deterministic function that decreases to zero as
\begin{equation}\label{deterministic}
C_t = \frac{\beta c}{(\beta + c)e^{\beta t/2}-c}, \quad \quad C_{0}=c.
\end{equation}
We will often write $C_t^c$ when it is useful to emphasize the starting point of $C_{t}$ at time 0. The process $C_{t}$ can be interpreted as a continuous-path time reversal of the deterministic process that governs the sample size. See \cite{PRS16}, Section 1.3 for a discussion. We will denote by $p^c_{\nn,\mm}(t)$ the transition probabilities of $\MM_{t}$, which depend on \eqref{deterministic} through $c$  (cf.~\cite{PRS16}, Section 4.3).  
One implication of the duality relation in \eqref{K-CIR duality}, given in the following lemma, makes the parallel with Lemma \ref{lemma: predictive likelihood} in the WF case, and provides the predictive law of observations given a past state of the signal.

\begin{lemma}
\label{lemma: cost to go_CIR}
Let $\YY_{i+1}$ be the vector of observations collected at time $t_{i+1}$ from model \eqref{K-CIR HHM}, with multiplicities $\NN_{i+1}$ and cardinality $c_{i+1}$. Then
\begin{equation}\nonumber
p(\YY_{i+1}|\zz_{t_{i}})
=r_{c_{i+1}}(\NN_{i+1}) 
\sum_{\nn\le \NN_{i+1}}p^{c_{i+1}}_{\NN_{i+1},\nn}(t)h(\zz_{t_{i}},\nn,C_{t}).
\end{equation} 
\end{lemma}
\begin{proof}
Using \eqref{h_functionCIR} we have
\begin{equation}\nonumber
\begin{aligned}
p(\YY_{i+1}|\zz_{t_{i}})
=&\,\int p(\yy_{i+1}|\zz_{t_{i+1}})p_{t_{i+1}-t_{i}}(\zz_{t_{i+1}}|\zz_{t_{i}})\d \zz_{t_{i+1}}\\
=&\,r_{c_{i+1}}(\NN_{i+1})\int h(\zz_{t_{i+1}},\NN_{i+1},c_{i+1})p_{t_{i+1}-t_{i}}(\zz_{t_{i+1}}|\zz_{t_{i}})\d \zz_{t_{i+1}}\\
=&\,r_{c_{i+1}}(\NN_{i+1})\E[h(\ZZ_{t_{i+1}},\NN_{i+1}, c_{i+1})| \ZZ_{t_{i}}=\zz_{t_{i}}]
\end{aligned}
\end{equation} 
from which \eqref{K-CIR duality} gives the result.
\end{proof}

The next Proposition corresponds to Proposition \ref{prop: nonpar backward} in the FV case and shows that the backward propagation of a conditional DW state yields the same as a forward propagation. 

\begin{proposition}
\label{prop: signal prediction CIR}
Assume \eqref{DW HMM}, let $y_{1},\ldots,y_{K}$ be distinct values and 
$\nn\in\Z_{+}^{K}$. Then {$\B_{t}(\Gamma^{\beta+c}_{\alpha+\sum_{j = 1}^Kn_j\delta_{y_j}}) 
	= \F_{t}( \Gamma^{\beta+c}_{\alpha+\sum_{j = 1}^Kn_j\delta_{y_j}})$}
with $\F_{t},\B_{t}$ as in \eqref{forward operator}-\eqref{backward operator}, 
and in particular
\begin{equation}\label{backward propagation DW}
\B_{t}\Big(\Gamma^{\beta+c}_{\alpha+\sum_{j = 1}^Kn_j\delta_{y_j}}\Big) 
=  \sum_{\oo \leq \kk \leq \nn} p^{c}_{\nn, \kk}(t)\Gamma^{\beta+C_t}_{\alpha+\sum_{j = 1}^Kk_j\delta_{y_j}}.
\end{equation} 
\end{proposition}
\begin{proof}
Let  $\YY_{i-1},\YY_{i+1}$ be vectors sampled at times $t_{i-1},t_{i+1}$, with multiplicities $\NN_{i-1}, \NN_{i+1}$ and cardinalities $(c_{i-1},c_{i+1})$, respectively. Then we have
\begin{equation}\label{forward prediction}
\begin{aligned}
p (\zz_{t_{i}} | \YY_{i-1} ) 
=&\, p(\zz_{t_{i}})\E \left[h(\zz_{t_{i}}, \MM_{t_{i}}, C_{t_{i+1}})  |  \MM_{t_{i-1}} = \NN_{i-1}, C_{t_{i-1}} = c_{i-1} \right] \\
=&\,\sum_{\oo\le \kk \le \NN_{i-1}}p^{c_{i-1}}_{\NN_{i-1},\kk}(t_{i}-t_{i-1})\Ga(\aa+\kk,\beta+C_{t_{i}-t_{i-1}})(\zz_{t_{i}})
\end{aligned}
\end{equation} 
where the first equality comes from Proposition 2.2 of \cite{PR14} and the second from Proposition 4.2 of \cite{PRS16} and the time homogeneity of $C_{t}$. 
Since 
by \eqref{update operator} {$p(\zz_{t_{i}}|\YY_{i+1}) = p(\YY_{i+1} | \zz_{t_{i}})p(\zz_{t_{i}})/r(\NN_{i+1},c_{i+1}),$}
we have 
\begin{equation}\nonumber
\begin{aligned}
p(\zz_{t_{i}}|\YY_{i+1})
=
&\,\frac{p(\zz_{t_{i}})}{r(\NN_{i+1},c_{i+1})}r(\NN_{i+1}, c_{i+1}) 
\sum_{\oo\le \kk\le \NN_{i+1}}p^{c_{i+1}}_{\NN_{i+1},\kk}(t_{i+1}-t_i)h(\zz_{i},\kk,C_{t})\\
=&\,
\sum_{\oo\le\kk\le \NN_{i+1}}p^{c_{i+1}}_{\NN_{i+1},\kk}(t_{i+1}-t_{i})\Ga(\aa+\kk,\beta+C_{t_{i+1}-t_{i}})(\zz_{t_{i}})
\end{aligned}
\end{equation} 
where we have used Lemma \ref{lemma: cost to go_CIR} and then \eqref{h p = post}. The rest of the proof now proceeds along the same lines of Proposition \ref{prop: nonpar backward}, with Theorem 3.2 in \cite{PRS16} in place of Theorem 3.1.
\end{proof}


\subsection{Smoothing distributions}

We proceed similarly to Section \ref{subsec: FV smoothing}, by first providing the smoothing distribution conditional on observations collected at adjacent times, then giving a full description of the mixture weights, and finally showing how the general expression of Theorem \ref{summary theorem} can be computed. We denote by $\YY_{i-1},\YY_i,\YY_{i+1}$ vectors of draws collected at times $t_{i-1}$, $t_i$, $t_{i+1}$ respectively, each defined as in \eqref{DW HMM} and the preceding discussion, with cardinalities $c_{i-1},c_{i},c_{i+1}$ and associated multiplicities $\NN_{i-1},\NN_{i},\NN_{i+1}$ for the distinct values $(y_1, \dots , y_K)$ observed overall.

\begin{theorem}\label{one-step-smoothingDW}
Under model \eqref{DW HMM}, let $\YY_{i-1},\YY_i,\YY_{i+1}$ be as above. Then there exist weights\\ $\{w_{\kk_{i-1},\NN_{i}, \kk_{i+1}}(\Delta_{i}, \Delta_{i+1}),\kk_{i-1}\le\NN_{i-1},\kk_{i+1}\le \NN_{i+1}\}$ which sum up to one such that
%
\begin{equation}\label{one step DW}
\begin{aligned}
\L(Z_{i}&\,| \YY_{i-1}, \YY_i, \YY_{i+1})=\\
=&\,\sum_{\oo \leq \kk_{i-1} \leq \NN_{i-1}}\sum_{\oo \leq \kk_{i+1} \leq \NN_{i+1}}v^{\NN_{i-1},\NN_{i+1}}_{\kk_{i-1},\NN_{i}, \kk_{i+1}}(\Delta_{i}, \Delta_{i+1})\Gamma^{\beta(\Delta_{i},\Delta_{i+1})}_{\alpha_{\kk_{i-1}+\NN_{i-1}+\kk_{i+1}}}.
\end{aligned}
\end{equation} 
where $\alpha_{\kk_{i-1}+\NN_{i-1}+\kk_{i+1}}$ is as in \eqref{alpha conditional}, $\Delta_{i} = t_{i}-t_{i-1}, \Delta_{i+1} = t_{i+1}-t_{i}$, 
\begin{equation}\label{beta con due delta}
\beta(\Delta_{i},\Delta_{i+1})=\beta+C^{c_{i-1}}_{\Delta_{i}}+c_{i}+C^{c_{i+1}}_{\Delta_{i+1}}
\end{equation}  
and $C^c_{t}$ is as in \eqref{deterministic}.
\end{theorem}
\begin{proof}
Throughout the proof we use the same notational conventions used in the proof of Theorem \ref{FV one-step-smoothing} with the obvious modifications due to the present setting. In particular, considering the projected state $Z_1(\mathcal{A}) | \YY_0(\mathcal{B}_n),\YY_1(\mathcal{B}_n), \YY_2(\mathcal{B}_n)$ conditional on binning the observations in a  partition $\B_{n}$ finer than $\A$, we have by the martingale convergence theorem that $Z_1(\mathcal{A}) | \YY_0(\mathcal{B}_n),\YY_1(\mathcal{B}_n), \YY_2(\mathcal{B}_n) $ converges weakly to
$Z_1(\mathcal{A}) | \YY_0, \YY_1, \YY_2$. An argument similar to that used in the proof of Theorem \ref{FV one-step-smoothing}, which makes use of Lemma \ref{lemma: cost to go_CIR} and \eqref{h_productCIR}, can now be used for characterising the conditional law of the DW states by means of the collection of predictive distributions. 
Specifically, this yields
\begin{equation}\label{K-CIR smoothing}
\begin{aligned}
\L  ( &\,\ZZ_1(\A) | \YY_{0,n},\YY_{1,n},\YY_{2,n}) = \\
=&\, \sum_{\oo\le\kk_{i-1}\le \NN_{i-1}}\sum_{\oo \le \kk_{i+1} \le \NN_{i+1}}p^{c_{i-1}}_{\NN_{i+1},\kk_{i-1}}p^{c_{i+1}}_{\NN_{i+1},\kk_{i+1}}
\\
&\,\times\frac
{{r^{(n)} ( \kk_{i-1} + \NN_{i}+\kk_{i+1}, C^{c_{i-1}}_{\Delta_{i}}+c_{i}+C^{c_{i+1}}_{\Delta_{i+1}}  )}}
{C'_{n}\, r^{(n)} (\kk_{i-1}, C^{c_{i-1}}_{\Delta_{i}} )r^{(n)}(\NN_{i}, c_{i})r^{(n)} (\kk_{i+1}, C^{c_{i+1}}_{\Delta_{i+1}} )} \\
&\,\times \Ga (\alpha(\A)+ \kk_{i-1}(\A)+\NN_{i}+ \kk_{i+1}(\A),\beta+C^{c_{i+1}}_{\Delta_{i}}+c_{i}+C^{c_{i-1}}_{\Delta_{i+1}})
\end{aligned}
\end{equation} 
where $C_{n}'$ is normalizing constant and $r^{(n)}$  is the marginal distribution in Lemma \ref{lemma_marginalDW} (with the notational convention here that $r(\cdot,c):=r_{c}(\cdot)$), relative to the model induced by the partition $\B_{n}$. Similarly to the FV case, the proof is now completed by observing that the limiting weights do not depend on the partition $\B_{n}$ and the previous display coincides with the projection onto the partition $\A$ of a finite mixture of laws of Gamma processes. 
\end{proof}

We next provide a full description of the weights appearing in the statement of Theorem \ref{one-step-smoothingDW}, with the 
convention 
$\gamma(\cdot,c):=\gamma_{c}(\cdot)$, for $\gamma_{c}(\cdot)$ as in \eqref{gamma_function}.

\begin{proposition}\label{weightsDW}
In the setting of Theorem \ref{one-step-smoothingDW}, we have
\[
v^{\NN_{i-1},\NN_{i+1}}_{\kk_{i-1}, \NN_{i}, \kk_{i+1}}(\Delta_{i},\Delta_{i+1}) = \frac{\gamma (|\kk_{i-1}+\NN_{i}+\kk_{i+1}|, C_{\Delta_{i}}+c_{i}+C_{\Delta_{i+1}}  )}{\gamma (|\kk_{i-1}|, C_{\Delta_{i}}  )\gamma (|\NN_{i}|, c_{i}  )\gamma (|\kk_{i+1}|, C_{\Delta_{i+1}}  )}w^{\NN_{i-1},\NN_{i+1}}_{\kk_{i-1}, \NN_{i}, \kk_{i+1}}
\]
where $w^{\NN_{i-1},\NN_{i+1}}_{\kk_{i-1}, \NN_{i}, \kk_{i+1}}$ as in Proposition \ref{FV smoothing weights}.
\end{proposition}
\begin{proof}
The result follows from \eqref{marginal_equality} with an easy calculation.
\end{proof}

Similarly to what we have done for the FV case, we can now revisit the second statement of Theorem \ref{summary theorem} in light of the above findings. We leave to the reader the adaptation of the argument preceding Theorem \ref{thm: general smoothing FV} to the present setting, which ultimately leads to
\begin{equation}\label{xi note 2}
\L  ( Z_{t_{i}} | \YY_{t_{0:i-1}}) = \sum_{\kk_{i-1} \leq \nminus_{i-1}}\rho_{1,\kk_{i-1}}\Gamma^{\beta+c_1}_{\alpha+\sum_{j = 1}^K k_{i-1,j}\delta_{y_j}},
\end{equation} 
and
\begin{equation}\label{xi two 2}
\L  ( Z_{t_{i}} | \YY_{t_{i+1:N}}  ) = \sum_{\kk_{i+1} \leq \nplus_{i+1}}\rho_{2,\kk_{i+1}}\Gamma^{\beta+c_2}_{\alpha+\sum_{j = 1}^K k_{i+1,j}\delta_{y_j}},
\end{equation} 
for some sets of weights $\{\rho_{1,\kk_{i-1}}\},\{\rho_{2,\kk_{i+1}}\}$, where $\nminus_{i-1},\nplus_{i+1}$ are here multiplicities associated to $\YY_{t_{0:i-1}},\YY_{t_{i+1:N}}$ respectively, relative to the distinct types $(y_{1},\ldots,y_{K})$ observed overall.

\begin{proposition}\label{thm: general smoothing DW}
Assume \eqref{xi note 2} and \eqref{xi two 2} hold, and let $\nn_{i}$ be the vector of multiplicities of $\YY_{t_{i}}^{1:c_{i}}$ relative to the $K$ distinct values in the whole dataset $\YY_{t_{0:N}}$. Then 
\begin{equation}\label{DW general smoothing two sums}
\L( Z_{t_{i}} | \YY_{t_{0:N}} ) =
\sum_{\kk_{i-1}\le \nminus_{i-1}}\sum_{\kk_{i+1}\le \nplus_{i+1}}
p_{\kk_{i-1},\nn_{i},\kk_{i+1}}
\Gamma^{\beta(\Delta_{i},\Delta_{i+1})}_{\alpha_{\kk_{i-1}+\nn_{i}+\kk_{i+1}}}
\end{equation} 
with $\alpha_{\kk_{i-1}+\nn_{i}+\kk_{i+1}}$ as in \eqref{alpha conditional} and  where the weights
\begin{equation}\label{DW pesi general smoothing two sums}
p_{\kk_{i-1}, \nn_{i},\kk_{i+1}}= 
\sum_{\hh\le \nminus_{i-1}:\,\hh \ge \kk_{i-1}}
\sum_{\ll\le \nplus_{i+1}:\, \ll \geq \kk_{i+1}}
\rho_{1,\hh}\rho_{2,\ll}v^{\hh, \ll}_{\kk_{i-1}, \nn_i,\kk_{i+1}},
\quad 
\kk_{i-1}\le \nminus_{i-1}, \kk_{i+1}\le \nplus_{i+1},
\end{equation} 
with $v^{\hh, \ll}_{\kk_{i-1}, \nn_i,\kk_{i+1}}$ as in Proposition \eqref{weightsDW}, sum up to one.

\end{proposition}
\begin{proof}
The proof is analogous to that of Proposition \ref{thm: general smoothing FV}.
\end{proof}


\subsection{Predictive distributions}

As a corollary to Proposition \ref{thm: general smoothing DW}, we can derive the predictive distribution of further samples collected at time $t_i$, given the original data set $\YY_{t_{0:N}}$. If the latter contains $c_{i}$ draws at time $t_{i}$, i.e., $\YY_{t_{i}} = (\YY_{t_{i}}^{1}, \dots , \YY_{t_{i}}^{c_i})$, denote for the present purposes by $\YY^{c_i+1}_{1:m}:=(Y^{c_i+1}_{1},\ldots,Y^{c_i+1}_{m})$ the elements of the $(c_{i}+1)$-th draw at $t_{i}$, where $m$ is random and we have dropped the subscript $t_{i}$. 

\begin{corollary}\label{DW predictive}
In the setting of Proposition \ref{thm: general smoothing DW}, let \eqref{DW general smoothing two sums} be the conditional law of $Z_{t_{i}}$ given $\YY_{t_{0:N}}$, and let 
\begin{equation}\label{predictiveDW_m}
m | \YY_{t_{0:N}} \sim 
\sum_{\kk_{i-1}\le \nminus_{i-1}}\sum_{\kk_{i+1}\le\nplus_{i+1}}
p_{\kk_{i-1}, \nn_{i}, \kk_{i+1}}\mathrm{NegBin} \left( \tilde\theta, \{1+\beta(\Delta_{i},\Delta_{i+1})\}^{-1} \right),
\end{equation}
with $\tilde\theta=\theta + |\kk_{i-1}| +|\nn_{i}|+ |\kk_{i+1}|$ and $\beta(\Delta_{i},\Delta_{i+1})$ as in \eqref{beta con due delta}.
Then the law of the $(k+1)$-st further element of the $(c_i+1)$-st draw from \eqref{DW HMM} is
\begin{equation}\label{predictiveDW_y}
\begin{aligned}
\mathbb{P}  (&\, Y^{c_i+1}_{k+1} \in A  |  \YY_{t_{0:N}}, Y^{c_i+1}_{1:k},m)  =\\
=&\,
\sum_{\kk_{i-1}\le \nminus_{i-1}}\sum_{\kk_{i+1}\le \nplus_{i+1}}
\tilde p_{\kk_{i-1},\nn_{i},\kk_{i+1}}
\frac{\alpha_{\kk_{i-1}+\nn_{i}+\kk_{i+1}}(A)
+ \sum_{j=1}^{k}\delta_{Y^{c_i+1}_j}(A)}
{\theta+ |\kk_{i-1}|+|\nn_{i}| + |\kk_{i+1}|+k}
\end{aligned}
\end{equation}
for every $A$ Borel set of $\Y$, with $\alpha_{\kk_{i-1}+\nn_{i}+\kk_{i+1}}$ as in \eqref{alpha conditional}, 
\begin{equation}\nonumber
\tilde p_{\kk_{i-1},\nn_{i},\kk_{i+1}} \propto 
p_{\kk_{i-1},\nn_{i},\kk_{i+1}}\, 
p(Y^{c_i+1}_{1:k} |\kk_{i-1},\nn_{i},\kk_{i+1})
p(m |\, |\kk_{i-1}|,|\nn_{i}|,|\kk_{i+1}|)
\end{equation} 
where $p_{\kk_{i-1},\nn_{i},\kk_{i+1}}$ are as in \eqref{DW pesi general smoothing two sums}, and $p(m |\, |\kk_{i-1}|,|\nn_{i}|,|\kk_{i+1}|)$, $p(Y^{c_i+1}_{1:k} |\kk_{i-1},\nn_{i},\kk_{i+1})$ are the distribution induced by \eqref{predictiveDW_m} and \eqref{predictiveDW_y} respectively.
\end{corollary}
\begin{proof}
Following \eqref{PP1}, with $M$ the random number of elements of $\YY^{c_i+1}$, we can write 
\[
\begin{aligned}
\mathbb{P}( &\,M = m  |  \YY_{t_{0:N}} ) = \E 
\Big[ \mathbb{P}( M = m  |  Z_{t_i} )  |  \YY_{t_{0:N}}\Big] 
= \E \bigg[\frac{e^{-|Z_{t_i}|}|Z_{t_i}|^m}{m!}  \Big\vert  \YY_{t_{0:N}}\bigg] \\
& = 
\sum_{\kk_{i-1}\le \nminus_{i-1}}\sum_{\kk_{i+1}\le\nplus_{i+1}}
p_{\kk_{i-1},\nn_{i},\kk_{i+1}}
\int \frac{e^{-z}z^m}{m!}\frac{\beta(\Delta_{i},\Delta_{i+1})^{\tilde \theta}}{\Gamma(\tilde \theta)}z^{\tilde \theta-1}e^{-\beta(\Delta_{i},\Delta_{i+1})\,z}\, \d z \\
& = \sum_{\kk_{i-1}\le \nminus_{i-1}}\sum_{\kk_{i+1}\le \nplus_{i+1}}
p_{\kk_{i-1},\nn_{i},\kk_{i+1}}
\mathrm{NegBin} \Big( \tilde \theta, \{1+\beta(\Delta_{i},\Delta_{i+1}) \}^{-1} \Big)
\end{aligned}
\] 
where we have used Proposition \ref{thm: general smoothing DW}, with $\tilde \theta = |\alpha_{\kk_{i-1}+\nn_{i}+\kk_{i+1}}|$ and $\beta(\Delta_{i},\Delta_{i+1})$ as in \eqref{beta con due delta}, from which the first statement follows (cf.~also Lemma \ref{lemma_marginalDW}). Notice that, as desired,
\[
|\alpha_{\kk_{i-1}+\nn_{i}+\kk_{i+1}}| = \alpha_{\kk_{i-1}+\nn_{i}+\kk_{i+1}}(\Y) = \theta +|\kk_{i-1}+\nn_{i}+\kk_{i+1}| = \theta + |\kk_{i-1}| +|\nn_{i}|+ |\kk_{i+1}|.
\]
Thus, again from \eqref{PP1}, conditional on $M = m$, the $m$ elements of $\YY^{c_i+1}$ are sampled from the marginal distribution induced by the normalized measure $Z_{t_i}/|Z_{t_i}|$. The second statement then follows along the same lines of Corollary \ref{cor:FV predictive}, upon noting that $\Gamma^\beta_\alpha/|\Gamma^\beta_\alpha| \overset{\d}{=} \Pi_\alpha$ for every finite measure $\alpha$.
\end{proof}
Similar implications to those discussed at the end of Section \ref{subsec: FV pred} apply here as well.


%




\end{document}